\documentclass[10pt, a4paper, reqno, oneoside]{amsart}

\usepackage{amsmath, amsfonts, amssymb, amscd, enumerate, scalefnt}

\usepackage{etex}
\usepackage{hyperref}
\usepackage{t1enc}
\usepackage[latin1]{inputenc}
\usepackage{graphicx}
\usepackage[all]{xy}
\usepackage{color}
\usepackage{slashed}
\usepackage[mathscr]{euscript}
\usepackage{enumitem}
\setlist[itemize]{noitemsep, topsep=1 pt}
\newcommand\bcdot{\ensuremath{
  \mathchoice
   {\mskip\thinmuskip\lower0.2ex\hbox{\scalebox{1.6}{$\cdot$}}\mskip\thinmuskip}}
   {\mskip\thinmuskip\lower0.2ex\hbox{\scalebox{1.6}{$\cdot$}}\mskip\thinmuskip}
   {\lower0.3ex\hbox{\scalebox{1.2}{$\cdot$}}}
   {\lower0.3ex\hbox{\scalebox{1.2}{$\cdot$}}}
}
\theoremstyle{plain}
\newtheorem{theo}{Theorem}[section]
\newtheorem{lem}[theo]{Lemma}
\newtheorem{prop}[theo]{Proposition}

\theoremstyle{definition}

\newtheorem{example}[theo]{Example}
\newtheorem{definition}[theo]{Definition}

\theoremstyle{plain}
\newtheorem{lemma}[theo]{Lemma}
\newtheorem{theorem}[theo]{Theorem}

\theoremstyle{definition}

\newtheorem{remark}[theo]{Remark}

\theoremstyle{plain}
\newtheorem{thmint}{Theorem}

\newtheorem{corint}[thmint]{Corollary}

\newtheorem{propint}[thmint]{Proposition}

\renewcommand{\=}{:=}

\renewcommand{\a}{\alpha}
\renewcommand{\b}{\beta}

\renewcommand{\d}{\delta}
\newcommand{\e}{\varepsilon}
\newcommand{\f}{\varphi}
\newcommand{\g}{\gamma}

\renewcommand{\k}{\kappa}
\renewcommand{\l}{\lambda}

\newcommand{\z}{\zeta}

\renewcommand{\S}{\Sigma}

\renewcommand{\L}{\Lambda}

\newcommand{\bR}{\mathbb{R}}

\newcommand{\bN}{\mathbb{N}}

\newcommand{\ga}{\mathfrak{a}}

\newcommand{\gc}{\mathfrak{c}}

\renewcommand{\gg}{\mathfrak{g}}
\newcommand{\gh}{\mathfrak{h}}
\newcommand{\gk}{\mathfrak{k}}
\newcommand{\gl}{\mathfrak{l}}
\newcommand{\gm}{\mathfrak{m}}
\newcommand{\gn}{\mathfrak{n}}

\newcommand{\gp}{\mathfrak{p}}

\newcommand{\gt}{\mathfrak{t}}

\newcommand{\gz}{\mathfrak{z}}

\newcommand{\so}{\mathfrak{so}}
\newcommand{\su}{\mathfrak{su}}

\newcommand{\fG}{\mathsf{G}}
\newcommand{\fH}{\mathsf{H}}
\newcommand{\fK}{\mathsf{K}}
\newcommand{\fL}{\mathsf{L}}
\newcommand{\fN}{\mathsf{N}}

\newcommand{\fGL}{\mathsf{GL}}
\newcommand{\fSL}{\mathsf{SL}}
\newcommand{\fO}{\mathsf{O}}
\newcommand{\fSO}{\mathsf{SO}}
\newcommand{\fU}{\mathsf{U}}
\newcommand{\fSU}{\mathsf{SU}}

\newcommand\SO{\mathrm{SO}}

\newcommand{\Sym}{\mathrm{Sym}}

\newcommand{\cC}{\mathcal{C}}

\newcommand{\eB}{\EuScript{B}}

\newcommand{\eF}{\EuScript{F}}

\newcommand{\eM}{\EuScript{M}}

\newcommand{\eS}{\EuScript{S}}

\newcommand{\eU}{\EuScript{U}}

\newcommand{\sF}{\mathscr{F}}

\newcommand{\sM}{\mathscr{M}}

\newcommand{\sU}{\mathscr{U}}

\newcommand{\sW}{\mathscr{W}}

\newcommand{\la}{\langle}
\newcommand{\ra}{\rangle}
\newcommand{\rar}{\rightarrow}

\renewcommand{\square}{\kern1pt\vbox
{\hrule height 0.6pt\hbox{\vrule width 0.6pt\hskip 3pt \vbox{\vskip
6pt}\hskip 3pt\vrule width 0.6pt}\hrule height0.6pt}\kern1pt}
\renewcommand{\=}{\  \raisebox{0.15mm}{:} {=} \ }
\newcommand{\rank}{\operatorname{rank}}

\DeclareMathOperator\Tr{Tr}
\DeclareMathOperator\Lie{Lie}
\DeclareMathOperator\ric{ric}
\DeclareMathOperator\scal{scal}

\DeclareMathOperator\Ad{Ad}
\DeclareMathOperator\ad{ad}
\DeclareMathOperator\vol{vol} 
\DeclareMathOperator\Id{Id}

\DeclareMathOperator{\vspan}{span}

\newcommand\Ric{\operatorname{Ric}}
\newcommand\Rm{\operatorname{Rm}}

\newcommand{\wt}{\widetilde}

\newcommand{\ol}{\overline}
\newcommand{\zero}{\operatorname{o}}
\def\<#1,#2>{\langle\,#1,\,#2\,\rangle}

\newcommand{\aac}{\`a}
\newcommand{\Aac}{\`A}

\newcommand{\Math}{{\it Mathematica\raise5 pt\hbox{$\scriptscriptstyle \circledR$}7}}
\newcommand{\n}{\nabla}

\newcommand{\beq}{\begin{equation}}
\newcommand{\eeq}{\end{equation}}
\def\<#1,#2>{\langle\,#1,\,#2\,\rangle}
\newcommand{\arr}{\begin{array}{rlll}}
\newcommand{\ea}{\end{array}}
\newcommand{\bea}{\begin{eqnarray}}
\newcommand{\eea}{\end{eqnarray}}
\newcommand{\bean}{\begin{eqnarray*}}
\newcommand{\eean}{\end{eqnarray*}}

\def\sideremark#1{\ifvmode\leavevmode\fi\vadjust{
\vbox to0pt{\hbox to 0pt{\hskip\hsize\hskip1em
\vbox{\hsize3cm\tiny\raggedright\pretolerance10000
\noindent #1\hfill}\hss}\vbox to8pt{\vfil}\vss}}}

\newcounter{ssig}
\newcounter{ttig}
\title[Diverging sequences of unit volume invariant metrics with bounded curvature]{Diverging sequences of unit volume invariant metrics \\ with bounded curvature}

\author{Francesco Pediconi}

\thanks{This work was supported by GNSAGA of INdAM} 

\begin{document}
\begin{abstract}
We study 1-parameter families in the space $\eM^{\fG}_1$ of $\fG$-invariant, unit volume metrics on a given compact, connected, almost-effective homogeneous space $M=\fG/\fH$. In particular, we focus on diverging sequences, i.e. that are not contained in any compact subset of $\eM^{\fG}_1$, and we prove some structure results for those which have bounded curvature. We also relate our results to an algebraic version of collapse. \end{abstract}

\maketitle


\section{Introduction} \setcounter{equation} 0

Given a compact, connected smooth manifold $M^m$ acted transitively and almost effectively by a compact Lie group $\fG$, the space $\eM^{\fG}$ of $\fG$-invariant Riemannian metrics on $M$ endowed with its standard $L^2$-metric $\la\,\cdot\,,\,\cdot\,\ra$ is a (finite dimensional) Riemannian symmetric space with non positive sectional curvature, and the subset $\eM^{\fG}_1 \subset \eM^{\fG}$ of unit volume $\fG$-invariant metrics is a totally geodesic submanifold (see e.g. \cite[Sec 4.1]{B\"o1}). We denote by $\fH$ the isotropy subgroup of $\fG$ at some distinguished point $x_{\zero} \in M$. \smallskip

It is well known that $\fG$-invariant unit volume Einstein metrics on $M$ can be characterized variationally as the critical points of the scalar curvature functional $\scal: \eM^{\fG}_1 \rar \bR$. In \cite{BWZ}, with the aim of searching for general saddle points, the authors proved that the functional $\scal$ satisfies the {\it Palais-Smale condition} on the subsets $(\eM^{\fG}_1)_{\e} \=\{g \in \eM^{\fG}_1 : \scal(g)\geq \e \}$, with $\e>0$. Namely, if $(g^{(n)}) \subset \eM^{\fG}_1$ is a sequence for which $\scal(g^{(n)}) \rar \e$ and $\big|\Ric^{\zero}(g^{(n)})\big|_{g^{(n)}} \rar 0$, then one can extract a subsequence which converges in the $\cC^{\infty}$-topology  to an Einstein metric $g^{(\infty)} \in \eM^{\fG}_1$ with $\scal(g^{(\infty)}) = \e>0$ \cite[Thm A]{BWZ}. Here, $\Ric^{\zero}(g^{(n)})$ is the traceless Ricci tensor of $g^{(n)}$ and $|\cdot|_{g^{(n)}}$ is the norm induced by $g^{(n)}$ on the tensor bundle over $M$. As is well known, the traceless Ricci tensor is precisely the negative gradient vector of the functional $\scal$ with respect to the standard $L^2$-metric $\la\,\cdot\,,\,\cdot\,\ra$.

On the other hand, again in \cite{BWZ}, the authors also studied the so called {\it $0$-Palais-Smale sequences}, i.e. $(g^{(n)}) \subset \eM^{\fG}_1$ such that $\scal(g^{(n)}) \rar 0$ and $\big|\Ric^{\zero}(g^{(n)})\big|_{g^{(n)}} \rar 0$. Notice that, unlike the previous case, a $0$-Palais-Smale sequence $(g^{(n)})$ cannot have convergent subsequences if $M$ is not a torus. This means that $(g^{(n)})$ goes off to infinity on the set $\eM^{\fG}_1$ and consequently we say that such sequences are {\it divergent}. Remarkably, there are topological obstructions on the existence of $0$-Palais-Smale sequences. In fact by \cite[Thm 2.1]{BWZ} if $M$ admits a $0$-Palais-Smale sequence, then there exists a closed, connected intermediate subgroup $\fH^{\zero} \subsetneq \fK^{\zero} \subset \fG^{\zero}$ such that the quotient $\fK^{\zero}/\fH^{\zero}$ is a torus. Here, $\fH^{\zero}$ and $\fG^{\zero}$ denote the identity components of $\fH$ and $\fG$, respectively.

This last theorem is optimal if the isotropy group $\fH$ is connected. In case $\fH$ is disconnected, the authors conjectured that $\fG/\fH$ is itself a homogeneous torus bundle \cite[p. 697]{BWZ}. \smallskip

The first main result proved in this paper for the purpose of generalizing \cite[Thm 2.1]{BWZ} is 

\begin{thmint} Let $M^m=\fG/\fH$ be a compact, connected homogenous space. If there exists a diverging sequence $(g^{(n)}) \subset \eM^{\fG}_1$ with bounded curvature, i.e. with $|\sec(g^{(n)})| \leq C$ for some constant $C>0$, then there exists an intermediate closed subgroup $\fH\subsetneq \fK \subset \fG$ such that the quotient $\fK/\fH$ is a torus. \label{MAIN} \end{thmint}

We stress that the proof of Theorem \ref{MAIN} is purely algebraic and constructive. In fact, we show that the sum of the eigenspaces associated to all the {\it shrinking eigenvalues} of any diverging sequence $(g^{(n)}) \subset \eM^{\fG}_1$ with bounded curvature is a reductive complement of $\gh = \Lie(\fH)$ into an intermediate $\Ad(\fH)$-invariant Lie subalgebra $\gh \subsetneq \gl \subsetneq \gg =\Lie(\fG)$, which uniquely detects a strictly intermediate Lie subgroup $\fH \subsetneq \fL \subsetneq \fG$, possibly not closed, such that the quotient $\ol{\fL}/\fH$ is a torus. Clearly Theorem \ref{MAIN} follows by setting $\fK\=\ol{\fL}$. Actually, we know more about the structure of any such a sequence: $(g^{(n)})$ approaches asymptotically, in a precise sense, a submersion-type metric with respect to the (locally) homogeneous fibration $\fL/\fH \rar \fG/\fH \rar \fG/\fL$ whose fibers shrink as $n\to+\infty$. We refer to Theorem \ref{main2} for more details.

Let us also remark that in \cite{BLS} the following estimate was proved: there exists a uniform constant $C>0$ which depends only on the dimension $m \in \bN$ such that \beq |\Rm(g)|_g \leq C |\Ric(g)|_g \quad \text{ for any $g \in \eM^{\fG}$ } \, , \label{gapthm} \eeq where $\Rm(g)$ denotes the curvature operator of $g$ \cite[Thm 4]{BLS}. This implies in particular that any sequence $(g^{(n)}) \subset \eM^{\fG}_1$ with $\scal(g^{(n)}) \rar \d \geq 0$ and $\big|\Ric^{\zero}(g^{(n)})\big|_{g^{(n)}} \rar 0$ has bounded curvature and hence, assuming that $M$ is not a torus, $0$-Palais-Smale sequences are special examples of diverging sequences with bounded curvature. Consequently, since we require neither that the Lie groups $\fH, \fG$ are connected, nor that the traceless Ricci goes to zero, Theorem \ref{MAIN} generalizes \cite[Thm 2.1]{BWZ}. We stress that this proves the previously mentioned conjecture in \cite[p. 697]{BWZ}. On the other hand, we point out that \cite[Thm 2.1]{BWZ} allows for changing the transitive group actions, while our Theorem \ref{MAIN} does not. \smallskip

Letting $\fN_{\fG}(\fH^{\zero})$ be the normalizer of $\fH^{\zero}$ in $\fG$, from Theorem \ref{MAIN} we immediately obtain the following

\begin{corint} If there exists no intermediate closed subgroup $\fH\subsetneq \fK \subset \fG$ such that the quotient $\fK/\fH$ is a torus, e.g. when $\rank(\fH)=\rank(\fN_{\fG}(\fH^{\zero}))$, then any diverging $1$-parameter family in $\eM^{\fG}_1$ has unbounded curvature. In particular, in such a case the scalar curvature functional satisfies the Palais-Smale condition on all of the space $\eM^{\fG}_1$. \label{MAINcor1} \end{corint}

We remark that, again by means of \eqref{gapthm}, $0$-Palais-Smale sequences get flatter and flatter as they go off to infinity. This last observation, together with the aim of providing an algebraic proof of the Palais-Smale condition for the functional $\scal$ (e.g. see \cite[Sec 2]{B\"o3} for an algebraic proof of the Bochner Theorem), brought us to study diverging sequences inside the subsets $(\eM^{\fG}_1)_{\e}$, with $\e>0$. The second main result proved in this paper is 

\begin{thmint} Let $M^m=\fG/\fH$ be a compact, connected homogenous space and let $\e>0$. Assume that there exists a diverging sequence $(g^{(n)}) \subset (\eM^{\fG}_1)_{\e}$ with bounded curvature and let $\fK$ be the intermediate closed subgroup determined by $(g^{(n)})$ as in Theorem \ref{MAIN}. Then, there exists a second intermediate closed subgroup $\fK \subsetneq \fK' \subset \fG$ such that the quotient $\fK'/\fH$ is not a torus. \label{MAIN2} \end{thmint}

As above, the proof of Theorem \ref{MAIN2} is is purely algebraic and constructive. In fact, we show that the sum of the eigenspaces associated to all the {\it generalized bounded eigenvalues} of any diverging sequence $(g^{(n)}) \subset (\eM^{\fG}_1)_{\e}$ with bounded curvature is a reductive complement of $\gh$ into a second intermediate $\Ad(\fH)$-invariant Lie subalgebra $\gh \subsetneq \gl \subsetneq \gl' \subsetneq \gg$, which uniquely detects a strictly intermediate Lie subgroup $\fL \subsetneq \fL' \subsetneq \fG$, possibly not closed, such that the quotient $\ol{\fL'}/\fH$ is not a torus. Again, Theorem \ref{MAIN2} follows by setting $\fK'\=\ol{\fL'}$. 

We also exhibit an example of a sequence of unit volume invariant metrics on the Stiefel manifold $V_3(\bR^5)=\fSO(5)/\fSO(2)$ which diverges with bounded curvature and whose scalar curvature converges to a positive constant. In that case, referring to the notation above, the intermediate subgroups are $\fL=\fK=\fSO(2){\times}\fSO(2)$ and $\fL'=\fK'=\fSO(4)$. We highlight here that, unlike the previous case, this example shows that a sequence $(g^{(n)}) \subset (\eM^{\fG}_1)_{\e}$ which diverges with bounded curvature does not necessarily approach asymptotically a submersion-type metric with respect to the (locally) homogeneous fibration $\fL'/\fH \rar \fG/\fH \rar \fG/\fL'$ given by the bigger Lie subgroup $\fL'$ (see Subsection \ref{V3}). \smallskip

Up to now, we still do not have an algebraic proof of the Palais-Smale condition for the scalar curvature functional on the subsets $(\eM^{\fG}_1)_{\e}$. We hope to consider this in a future paper. \smallskip

Finally, we relate our results on diverging sequences with bounded curvature to an algebraic version of collapse, which naturally arises in the study of equivariant convergence of locally homogeneous Riemannian spaces. We recall that a sequence $(g^{(n)}) \subset \eM^{\fG}$ is said to be {\it algebraically collapsed} if the norm of the bracket of the Lie algebra $\gg$ blows up along $(g^{(n)})$, that is $|\mu|_{Q_{\gh}+g^{(n)}} \rar +\infty$, where $\mu \in \L^2\gg^* \otimes \gg$ is just $\mu(X,Y) \= [X,Y]$ and $Q_{\gh}$ is any $\Ad(\fH)$-invariant inner product on $\gh$, which is needed to extend $g^{(n)}$ to the whole $\gg$. Geometrically, this condition is equivalent (see \cite[Sec 9]{BL}) to the existence of a sequence of $g^{(n)}$-Killing vector fields $X^{(n)}$ induced by the action of $\fG$ on $M$ such that $$\big|X^{(n)}_{x_{\zero}}\big|_{g^{(n)}} =1 \,\, , \quad \big|(\n^{g^{(n)}}\!X^{(n)})_{x_{\zero}}\big|_{g^{(n)}} \rar +\infty \,\, .$$ Roughly speaking it means that, up to normalize with respect to the $1$-jet norm, the sequence $(X^{(n)})$ is {\it running into the isotropy at $x_{\zero}$} as $n \rar +\infty$.

Of course algebraically collapsed sequences are necessarily divergent. Remarkably, the following weaker converse assertion follows from Theorem \ref{main2}.

\begin{propint} Let $M^m=\fG/\fH$ be a compact, connected homogenous space and suppose that $\pi_1(M)$ is finite. If $(g^{(n)}) \subset \eM^{\fG}_1$ is a diverging sequence with bounded curvature, then it is algebraically collapsed. \label{MAINcor2} \end{propint}

Notice that Proposition \ref{MAINcor2} is optimal. In fact, we provide an easy example of a sequence of unit volume invariant metrics on the product $S^1{\times}S^2$ which diverges with bounded curvature and is not algebraically collapsed. \smallskip

The paper is structured as follows. In Section \ref{section2}, we recall some basic properties of the space $\eM^{\fG}$ of $\fG$-invariant metrics and some well known formulas for the curvature of compact homogeneous Riemannian spaces, which will be needed afterwards. Section \ref{section3} is devoted to the study of $\fH$-subalgebras and submersion directions, which are of crucial importance in our interests. In Section \ref{section4}, we prove Theorem \ref{MAIN}, Theorem \ref{MAIN2} and we discuss an explicit example. In Section \ref{section5}, we briefly introduce the algebraic collapse and we prove Proposition \ref{MAINcor2}. Finally, in Appendix \ref{appendixA}, we provide a proof of a fundamental estimate, due to B\"ohm, which is needed in the proof of our main theorems, and we write down some computations related to the example that we saw in Section \ref{section4}. \medskip

\noindent{\it Acknowledgement.} This work has been set up during the visit of the author at WWU M\"unster. We warmly thank Christoph B\"ohm for his kind hospitality and for many fundamental discussions about several aspects of this paper. We are also grateful to Luigi Verdiani for his important suggestions. We thank Simon Lohove and Andrea Spiro for numerous pleasant conversations. Finally, we would like to thank the anonymous referee for his/her careful reading of the manuscript and useful comments. \smallskip

\section{Preliminaries and notation} \label{section2} \setcounter{equation} 0

\subsection{The space of $\fG$-invariant metrics} \hfill \par

Let $M=\fG/\fH$ be a compact, connected and almost effective $m$-dimensional homogeneous space, with $\fG$ and $\fH$ compact Lie groups. We fix once and for all an $\Ad(\fG)$-invariant Euclidean inner product $Q$ on the Lie algebra $\gg\=\Lie(\fG)$ and we indicate with $\gm$ the $Q$-orthogonal complement of $\gh\=\Lie(\fH)$ in $\gg$. From now on, we will always identify any $\fG$-invariant tensor field on $M$ with the corresponding $\Ad(\fH)$-invariant tensor on $\gm$ by the natural evaluation map at the point $e\fH \in M$. The restriction $Q_{\gm} \= Q|_{\gm\otimes\gm}$ of $Q$ on the complement $\gm$ defines a normal $\fG$-invariant metric on $M$. Up to a normalization we can assume that $\vol(Q_{\gm})=1$. We denote by $\eM^{\fG}$ the set of $\fG$-invariant metrics on $M$ and by $\eM^{\fG}_1$ the subset of unit volume ones. \smallskip

The set of inner products on $\gm$, which we indicate with $P(\gm)$, is an open cone in the space $\Sym(\gm,Q_{\gm})$ of symmetric endomorphism of $(\gm,Q_{\gm})$ by the embedding \beq g \,\longmapsto\, A_g \,\, , \quad g=Q_{\gm}(A_g \,\cdot\,,\,\cdot\,) \label{embPSym} \eeq and it is acted transitively by $\fGL(\gm)$ with isotropy in $Q_{\gm}$ isomorphic to $\fO(\gm,Q_{\gm})$. So it admits the coset space presentation $P(\gm)=\fGL(\gm)\big/\fO(\gm,Q_{\gm})$. It can also be endowed with the standard $\fGL(\gm)$-invariant Riemannian metric defined by \beq \la A_1,A_2\ra_g \= \Tr(A_g^{-1}A_1A_g^{-1}A_2) \quad \text{ for any $A_1,A_2 \in T_gP(\gm) \simeq \Sym(\gm,Q_{\gm})$ }\,\, . \label{stmet} \eeq Since the map $a \mapsto (a^T)^{-1}$ is an involutive automorphism of $\fGL(\gm)$ with fixed point set $\fO(\gm,Q_{\gm})$, $P(\gm)$ is a Riemannian symmetric space. The space $\eM^{\fG}$ is nothing but the fixed point set of the isometric action of $\fH$ on $P(\gm)$ given by \beq A_g \longmapsto (\Ad(h)|_{\gm})A_g(\Ad(h)|_{\gm})^T \,\, , \quad h \in \fH \, , \,\, g \in P(\gm) \label{Ad(H)-action} \eeq and so $\eM^{\fG}$ is a totally geodesic submanifold of $P(\gm)$. Since $P(\gm)$ splits isometrically as $\bR \times \fSL(\gm)/\fSO(\gm,Q_{\gm})$ and $\fSL(\gm)/\fSO(\gm,Q_{\gm})$ is a symmetric space of non-compact type, we conclude that $\eM^{\fG}$ endowed with the restriction of \eqref{stmet} is a Riemannian symmetric space with non-positive sectional curvature. 

We consider now a $Q_{\gm}$-orthogonal, $\Ad(\fH)$-invariant irreducible decomposition \beq \gm=\gm_1+{\dots}+\gm_{\ell} \label{invdec} \,\, .\eeq If the adjoint representation of $\fH$ is {\it monotypic}, i.e. $\gm_i \not\simeq\gm_j$ for any $1\leq i < j \leq \ell$, the decomposition \eqref{invdec} is unique up to ordering and by the Schur Lemma any invariant metric $g \in \eM^{\fG}$ can be uniquely written as \beq g= \l_1Q_{\gm_1} + \dots + \l_{\ell}Q_{\gm_{\ell}} \,\, , \label{diag} \eeq where $Q_{\gm_i} \= Q|_{\gm_i \otimes \gm_i}$ and $\l_1,\dots,\l_{\ell} \in \bR$ are positive coefficients. In general, the decomposition \eqref{invdec} is not unique if some modules $\gm_i$ are equivalent to each other and the invariant metrics need not to be diagonal anymore. We denote by $\sF^{\fG}$ the space of ordered, $Q_{\gm}$-orthogonal, $\Ad(\fH)$-invariant, irreducible decompositions of $\gm$, which is itself a compact homogeneous space (see \cite[Lemma 4.19]{B\"o1}). \smallskip

The space $\eM^{\fG}$ can be described in terms of any fixed decomposition $\f \in \sF^{\fG}$. Instead of using such approach, we will allow the decomposition of $\gm$ to vary in the space $\sF^{\fG}$. In fact, it is known that for any $g \in \eM^{\fG}$, there exists $\f=(\gm_1,{\dots},\gm_{\ell}) \in \eF^{\fG}$ with respect to which $g$ is diagonal, i.e. takes the form \eqref{diag} (see see e.g. \cite[Sec 1]{WZ}). Any such $\f$ will be called a {\it good decomposition for $g$}. Notice that an invariant metric $g$ may admit more good decompositions. \smallskip

Since $\eM^{\fG}$ is a symmetric space with non-positive sectional curvature, by the Cartan-Hadamard Theorem its Riemannian exponential map is surjective. Moreover, by \eqref{embPSym} and \eqref{Ad(H)-action} $$T_{Q_{\gm}}\eM^{\fG} = \Sym(\gm,Q_{\gm})^{\Ad(\fH)} = \Big\{v \in \Sym(\gm,Q_{\gm}): (\Ad(h)|_{\gm}){\cdot}v{\cdot}(\Ad(h)|_{\gm})^T=v \text{ for any } h \in \fH \Big\} \,\, .$$ For any fixed $v \in T_{Q_{\gm}}\eM^{\fG}$ there exists a decomposition $\f=(\gm_1,{\dots},\gm_{\ell}) \in \sF^{\fG}$ such that $$v=v_1Q_{\gm_1}+{\dots}+v_{\ell}Q_{\gm_{\ell}} \quad \text{ for some } \,\, v_1,{\dots},v_{\ell} \in \bR \,\, .$$ By \cite[p. 226]{Hel} the geodesic $\g_v(t)$ in $\eM^{\fG}$ starting from $Q_{\gm}$ and tangent to $v \in T_{Q_{\gm}}\eM^{\fG}$, with respect to the same decomposition $\f$, takes the form \beq \g_v(t)=e^{tv_1}Q_{\gm_1}+{\dots}+e^{tv_{\ell}}Q_{\gm_{\ell}} \,\, . \eeq Any such decomposition will be called {\it good decomposition for $v$}. Notice that the eigenvalues $v_i$ do not depend on the choice of the good decomposition. Since $\vol(\g_v(t))=\exp(t\Tr(v))$, it follows that $\eM^{\fG}_1$ is a totally geodesic submanifold of $\eM^{\fG}$. In particular, we consider the unit tangent sphere \beq \S \= \Big\{v \in \Sym(\gm,Q_{\gm})^{\Ad(\fH)} : \, \Tr(v^2)=1 \, , \,\, \Tr(v)=0 \Big\} \label{Sigma} \eeq so that $$\eM^{\fG}_1=\{Q_{\gm}\} \cup \{\g_v(t): v \in \S \, ,\,\, t>0\} \,\, .$$ Notice that the space $\eM^{\fG}_1$ is a singleton if and only if $\fG/\fH$ is isotropy irreducible. In that case $\S=\emptyset$.

\subsection{Curvature of compact homogeneous Riemannian spaces} \hfill \par

Let us fix a decomposition $\f=(\gm_1,{\dots},\gm_{\ell}) \in \eF^{\fG}$ for the reductive complement $\gm$ and set $I \= \{1,{\dots},\ell\}$. Notice that the number $\ell$ of irreducible invariant submodules does not depend on the choice of the decomposition $\f$. We set $d_i \= \dim(\gm_i)$ which are again, up to ordering, independent of $\f$. A basis $(e_{\a})$ for $\gm$ is said to be {\it $\f$-adapted} if $$e_{1},{\dots},e_{d_1} \in \gm_1 \,\, , \quad e_{d_1+1},{\dots},e_{d_1+d_2} \in \gm_2 \,\, , \quad {\dots} \quad , \quad e_{d_1+{\dots}+d_{\ell-1}+1},{\dots},e_n \in \gm_{\ell} \,\, .$$ For any subset $I' \subset I$, we set \beq \gm_{I'} \= \sum_{i\in I'}\gm_i \,\, , \quad d_{I'} \= \sum_{i\in I'}d_i \,\, .\eeq Moreover, for any $I_1,I_2,I_3 \subset I$ we define \beq [I_1I_2I_3]_{\f} \= \sum_{\substack{e_{\a}\in \gm_{I_1} \\ e_{\b}\in \gm_{I_2} \\ e_{\g}\in \gm_{I_3}}} Q([e_{\a},e_{\b}],e_{\g})^2 \,\, , \label{coef1} \eeq where $(e_{\a})$ is a $Q_{\gm}$-orthonormal $\f$-adapted basis for $\gm$. If  at least one of the three index sets is a singleton, say e.g. $I_1=\{i\}$, we will shortly write $[iI_2I_3]_{\f}$ instead of $[\{i\}I_2I_3]_{\f}$. Notice that $[I_1I_2I_3]_{\f}$ is symmetric in all three entries and does not depend on the choice of the $Q_{\gm}$-orthonormal basis $(e_{\a})$. Furthermore, $[I_1I_2I_3]_{\f} \geq 0$ with $[I_1I_2I_3]_{\f} = 0$ if and only if $[\gm_{I_1},\gm_{I_2}] \cap \gm_{I_3}=\{0\}$. Finally, though the coefficients $[I_1I_2I_3]_{\f}$ do depend on the choice of $\f$, the correspondence $\f \rar [I_1I_2I_3]_{\f}$ is a continuous function on $\eF^{\fG}$ (see \cite[Sec 4.3]{B\"o1}). \smallskip

We introduce now the Casimir operator $$C_{Q_{\gh}}: \gm \rar \gm \,\, , \quad C_{Q_{\gh}} \= -\sum_i\ad(z_i) \circ \ad(z_i) \,\, ,$$ where $Q_{\gh}\=Q|_{\gh \otimes \gh}$ and $(z_i)$ is any $Q_{\gh}$-orthonormal basis for $\gh$. Then, the following conditions hold: \beq C_{Q_{\gh}}|_{\gm_i}=c_i\Id_{\gm_i} \,\, , \label{c_i} \eeq with $c_i\geq0$ and $c_i=0$ if and only if $[\gh,\gm_i]=\{0\}$ (see \cite[Sec 1]{WZ}). We also define the coefficients $b_1,{\dots},b_{\ell} \in \bR$ by setting \beq (-B)|_{\gm_i\otimes\gm_i} = b_iQ_{\gm_i} \,\, , \label{b_i} \eeq where $B$ is the Cartan-Killing form of $\gg$. Since $\gg$ is compact, it follows that $b_i \geq 0$ and $b_i=0$ if and only if $\gm_i \subset \gz(\gg)$. If $\fG$ is semisimple, then one can choose $Q=-B$, so that $b_i=1$ for any $i$. \smallskip

Notice that both the coefficients $c_i$ and $b_i$ do depend on the choice of $\f$, while \beq b_{\fG/\fH}\=\Tr_{Q_{\gm}}(-B)=\sum_{i \in I}d_ib_i \label{bG/H} \eeq does not. Moreover, they are related by the following useful relation (see \cite[Lemma 1.5]{WZ}):
\beq d_ib_i=2d_ic_i+\sum_{j,k \in I}[ijk]_{\f} \quad \text{ for any $i \in I$ } \,\, . \label{dbc} \eeq

Let now $g \in \eM^{\fG}$ be a diagonal metric as in \eqref{diag} with respect to $\f$. The next proposition gives explicit formulas for the sectional curvature $\sec(g)$ of $g$ along $\f$-adapted $2$-planes in $\gm$. Notice that one could obtain \eqref{sec1} and \eqref{sec2} from \cite[Cor 1.13]{GZ} where the authors proved a more general formula for the sectional curvature of diagonal cohomogeneity one metrics.

\begin{prop} Let $X,Y \in \gm$ be $Q_{\gm}$-orthonormal vectors. If $X \in \gm_i$ and $Y \in \gm_j$ for some $i,j \in I$, then the sectional curvature of $g$ along $X\wedge Y$ is given by \begin{align}
&\sec(g)(X{\wedge}Y) = \frac1{\l_i}\big|[X,Y]_{\gh}\big|_Q^2+\sum_{k\in I}\frac{4\l_i-3\l_k}{4\l_i^2}\big|[X,Y]_{\gm_k}\big|_Q^2 \,\, ,& \quad \text{ if } i=j \,\, , \label{sec1} \\
&\sec(g)(X{\wedge}Y) = \sum_{k\in I}\frac{\l_i^2+\l_j^2-3\l_k^2-2\l_i\l_j+2\l_i\l_k+2\l_j\l_k}{4\l_i\l_j\l_k}\big|[X,Y]_{\gm_k}\big|_Q^2 \,\, ,& \quad \text{ if } i\neq j \,\, . \label{sec2}
\end{align}
\end{prop}
\begin{proof}We put $\tilde{X} \=\frac1{\sqrt{\l_i}}X$, $\tilde{Y} \=\frac1{\sqrt{\l_j}}Y$. By \cite[Thm 7.30]{Bes} it holds that \beq \begin{aligned} \sec(X{\wedge}Y)= -\frac34\big|[\tilde{X},\tilde{Y}]_{\gm}\big|_g^2-\frac12g\big([\tilde{X},[\tilde{X},\tilde{Y}]]_{\gm},\tilde{Y}]\big)&-\frac12g\big([\tilde{Y},[\tilde{Y},\tilde{X}]]_{\gm},\tilde{X}]\big)+\\
&+\big|U^g(\tilde{X},\tilde{Y})\big|_g^2-g\big(U^g(\tilde{X},\tilde{X}),U^g(\tilde{Y},\tilde{Y})\big) \,\, ,\label{formulaBesse} \end{aligned} \eeq where $U^g:\gm\otimes\gm \rar \gm$ is the symmetric tensor uniquely defined by \beq 2g(U^g(X,Y),Z)\=g([Z,X]_\gm,Y)+g([Z,Y]_\gm,X) \,\, . \label{U}\eeq We observe that \begin{align}
&\big|[\tilde{X},\tilde{Y}]_{\gm}\big|_g^2= \sum_{k \in I}\frac{\l_k}{\l_i\l_j}\big|[X,Y]_{\gm_k}\big|_Q^2\,\, , \nonumber \\
&g([\tilde{X},[\tilde{X},\tilde{Y}]]_{\gm},\tilde{Y})=\frac1{\l_i}Q([X,[X,Y]],Y)=-\frac1{\l_i}\big|[X,Y]\big|_Q^2 \,\, , \label{1pezzo} \\
&g([\tilde{Y},[\tilde{Y},\tilde{X}]]_{\gm},\tilde{X})=\frac1{\l_j}Q([Y,[Y,X]],X)=-\frac1{\l_j}\big|[X,Y]\big|_Q^2 \,\, . \nonumber
\end{align} Let now $(e_{\a})$ be a $\f$-adapted $Q_{\gm}$-orthonormal basis for $\gm$. Then $$g(U^g(\tilde{X},\tilde{X}),e_{\a})=g([e_{\a},\tilde{X}],\tilde{X})=\frac1{\l_i}Q([X,X],e_{\a})=0$$ and so \beq U^g(\tilde{X},\tilde{X})=U^g(\tilde{Y},\tilde{Y})=0 \,\, . \label{2pezzo} \eeq Finally \begin{align}
|U^g(\tilde{X},\tilde{Y})|_g^2&=\sum_{k \in I}\sum_{e_{\a} \in \gm_k}g(U^g(\tilde{X},\tilde{Y}),\textstyle\frac1{\sqrt{\l_k}}e_{\a})^2 \nonumber \\
&=\sum_{k \in I}\sum_{e_{\a} \in \gm_k}\frac1{4\l_i\l_j\l_k}\Big(g([e_{\a},X],Y)+g([e_{\a},Y],X)\Big)^2 \label{3pezzo} \\
&=\sum_{k \in I}\frac{|\l_i-\l_j|^2}{4\l_i\l_j\l_k}\big|[X,Y]_{\gm_k}\big|_Q^2 \,\, . \nonumber
\end{align} By \eqref{1pezzo}, \eqref{2pezzo} and \eqref{3pezzo}, formula \eqref{formulaBesse} becomes \begin{align*}
\sec(g)(X{\wedge}Y)&= -\sum_{k \in I}\frac{3\l_k}{4\l_i\l_j}\big|[X,Y]_{\gm_k}\big|_Q^2+\frac12\Big(\frac1{\l_i}+\frac1{\l_j}\Big)\big|[X,Y]\big|_Q^2+\sum_{k \in I}\frac{|\l_i-\l_j|^2}{4\l_i\l_j\l_k}\big|[X,Y]_{\gm_k}\big|_Q^2 \\
&=\frac{\d_{ij}}{\l_i}\big|[X,Y]_{\gh}\big|_Q^2+\sum_{k 	\in I}\frac{2\l_i+2\l_j-3\l_k}{4\l_i\l_j}\big|[X,Y]_{\gm_k}\big|_Q^2+\sum_{k \in I}\frac{|\l_i-\l_j|^2}{4\l_i\l_j\l_k}\big|[X,Y]_{\gm_k}\big|_Q^2 \\
&=\frac{\d_{ij}}{\l_i}\big|[X,Y]_{\gh}\big|_Q^2+\sum_{k\in I}\frac{\l_i^2+\l_j^2-3\l_k^2-2\l_i\l_j+2\l_i\l_k+2\l_j\l_k}{4\l_i\l_j\l_k}\big|[X,Y]_{\gm_k}\big|_Q^2
\end{align*} and so both \eqref{sec1} and \eqref{sec2} follow. \end{proof}

As far as it concerns the Ricci tensor $\Ric(g): \gm \otimes \gm \rar \bR$, the following lemma holds true (see also \cite[Lemma 1.1]{PaSa}).

\begin{lemma} For any $1 \leq i \leq \ell$ it holds that \beq \Ric(g)|_{\gm_i \otimes \gm_i} = \l_i\ric_i(g)\,Q_{\gm_i} \,\, , \quad \ric_i(g) \= \frac{b_i}{2\l_i}-\frac1{2d_i}\sum_{j,k \in I}[ijk]_{\f}\frac{\l_k}{\l_i\l_j}+\frac1{4d_i}\sum_{j,k \in I}[ijk]_{\f}\frac{\l_i}{\l_j\l_k} \,\, . \label{ric} \eeq If the adjoint representation of $\fH$ on $\gm$ is monotypic, then the Ricci tensor decomposes as $$\Ric(g)=\l_1\ric_1(g)Q_{\gm_1}+{\dots}+\l_{\ell}\ric_{\ell}(g)Q_{\gm_{\ell}} \,\, .$$ \end{lemma}
\begin{proof} By the Schur Lemma, for any $1 \leq i \leq \ell$ there exist $x_i \in \bR$ such that $\Ric(g)|_{\gm_i \otimes \gm_i} =x_iQ_{\gm_i}$. Then, letting $(e_{\a})$ be a $\f$-adapted $Q_{\gm}$-orthonormal basis for $\gm$, it necessarily holds that \beq \ric_i(g)=\frac{x_i}{\l_i} = \frac1{d_i\l_i}\sum_{e_{\a} \in \gm_i}\Ric(g)(e_{\a},e_{\a}) = \frac1{d_i}\sum_{e_{\a} \in \gm_i}\Ric(g)\big({\textstyle\frac{e_{\a}}{\sqrt{\l_i}}},{\textstyle\frac{e_{\a}}{\sqrt{\l_i}}}\big) \,\, . \label{xili} \eeq Notice that, from \eqref{coef1}, \eqref{c_i} and the $\Ad(\fG)$-invariance of $Q$, we directly obtain that \beq \sum_{\substack{e_{\a} \in \gm_i \\ e_{\b} \in \gm_j}}\big|[e_{\a},e_{\b}]_{\gh}\big|_Q^2=\d_{ij}d_ic_i \,\, , \quad \sum_{\substack{e_{\a} \in \gm_i \\ e_{\b} \in \gm_j}}\big|[e_{\a},e_{\b}]_{\gm_k}\big|_Q^2=[ijk]_{\f} \,\, . \label{dcbra} \eeq Therefore for any fixed $i \in I$ we get \begin{align}
\sum_{j\in I}\sum_{\substack{e_{\a} \in \gm_i \\ e_{\b} \in \gm_j}}\sec(g)(e_{\a}\wedge e_{\b}) &\overset{{\color{white} (2.13)}}{=} \sum_{j \in I}\sum_{\substack{e_{\a} \in \gm_i \\ e_{\b} \in \gm_j}}\frac{\d_{ij}}{\l_i}\big|[e_{\a},e_{\b}]_{\gh}\big|_Q^2+ \nonumber \\
&\phantom{aaaaaaaa}+\sum_{j,k \in I}\sum_{\substack{e_{\a} \in \gm_i \\ e_{\b} \in \gm_j}}\frac{\l_i^2+\l_j^2-3\l_k^2-2\l_i\l_j+2\l_i\l_k+2\l_j\l_k}{4\l_i\l_j\l_k}\big|[e_{\a},e_{\b}]_{\gm_k}\big|_Q^2 \nonumber \\
&\overset{\eqref{dcbra}}{=}\frac{d_ic_i}{\l_i}+\frac14\sum_{j,k \in I}[ijk]_{\f}\frac{\l_i^2-(\l_j-\l_k)^2}{\l_i\l_j\l_k} \label{sumsecRic} \\
&\overset{\eqref{dbc}}{=}\frac{d_ib_i}{2\l_i}+\frac14\sum_{j,k \in I}[ijk]_{\f}\Big(-\frac{\l_j^2+\l_k^2}{\l_i\l_j\l_k}+\frac{\l_i}{\l_j\l_k}\Big) \nonumber\\
&\overset{{\color{white} (2.13)}}{=} \frac{d_ib_i}{2\l_i}-\frac12\sum_{j,k \in I}[ijk]_{\f}\frac{\l_k}{\l_i\l_j}+\frac14\sum_{j,k \in I}[ijk]_{\f}\frac{\l_i}{\l_j\l_k} \,\, . \nonumber
\end{align} Finally, from \eqref{xili} and \eqref{sumsecRic} we conclude that \begin{align*}
\ric_i(g) &= \frac1{d_i}\sum_{e_{\a} \in \gm_i}\Ric(g)\big({\textstyle\frac{e_{\a}}{\sqrt{\l_i}}},{\textstyle\frac{e_{\a}}{\sqrt{\l_i}}}\big) \\
&= \frac1{d_i}\sum_{j\in I}\sum_{\substack{e_{\a} \in \gm_i \\ e_{\b} \in \gm_j}}\sec(g)(e_{\a}\wedge e_{\b}) \\
&= \frac{b_i}{2\l_i}-\frac1{2d_i}\sum_{j,k \in I}[ijk]_{\f}\frac{\l_k}{\l_i\l_j}+\frac1{4d_i}\sum_{j,k \in I}[ijk]_{\f}\frac{\l_i}{\l_j\l_k} \,\, .
\end{align*} The last claim follows directly by applying the Schur Lemma. \end{proof}

Notice that the coefficients $\ric_i$ defined in \eqref{ric} are precisely the diagonal terms of the Ricci operator $\wt{\Ric}(g):\gm \rar \gm$ given by the relation $\Ric(g)(X,Y)=g\big(\wt{\Ric}(g)(X),Y\big)$.

Finally, by \eqref{ric} it comes that the scalar curvature of $g$ is given by (see also \cite[Sec 1]{WZ}) \beq \scal(g) = \sum_{i\in I}d_i\ric_i(g)= \frac12\sum_{i\in I}\frac{d_ib_i}{\l_i}-\frac14\sum_{i,j,k \in I}[ijk]_{\f}\frac{\l_i}{\l_j\l_k} \,\, .\label{scal} \eeq

\smallskip

\section{$\fH$-subalgebras, submersion metrics and submersion directions} \label{section3} \setcounter{equation} 0

\subsection{$\fH$-subalgebras} \hfill \par

We consider again a compact, connected and almost effective $m$-dimensional homogeneous space $M=\fG/\fH$, with $\fG$ and $\fH$ compact Lie groups, and a fixed $\Ad(\fG)$-invariant Euclidean inner product $Q$ on the Lie algebra $\gg\=\Lie(\fG)$. We highlight that we call {\it Lie subgroup of $\fG$} any immersed submanifold of $\fG$ which is also a subgroup. We refer to \cite{B\"o1,B\"o2} for what concerns $\fH$-subalgebras and submersion directions. \smallskip

Since $\fG$ is compact it is well known that $\gg$ is reductive, i.e. its radical coincides with its center $\gz(\gg)$. We observe also that every Lie subalgebra $\gk \subset \gg$ is reductive itself. This last claim can be easily proved by noticing that restriction of $Q$ to $\gk$ is an $\Ad(\fK^{\zero})$-invariant Euclidean inner product on $\gk$, where we indicated with $\fK^{\zero}$ the connected Lie subgroup of $\fG$ with Lie algebra $\gk$. Hence, any Lie subalgebra $\gk \subset \gg$ splits as $\gk=[\gk,\gk] \oplus \gz(\gk)$. We denote also by $\ol{\fK^{\zero}}$ the closure of $\fK^{\zero}$ in $\fG$, which is itself a Lie group, and by $\ol{\gk}$ its Lie algebra, which is called {\it Malcev-closure of $\gk$ in $\gg$} \cite[p. 51]{OV}. Then, $\ol{\gk}$ is a compact subalgebra of $\gg$, possibly $\ol{\gk}=\gg$, and moreover $[\ol{\gk},\ol{\gk}]=[\gk,\gk]$ by \cite[Thm 3, p. 52]{OV}. 

\begin{definition} A {\it $\fH$-subalgebra of $\gg$} is an $\Ad(\fH)$-invariant intermediate Lie subalgebra $\gk$ which lies properly between $\gh=\Lie(\fH)$ and $\gg$. An $\fH$-subalgebra $\gk$ is called {\it toral} if $[\gk,\gk] \subset \gh$, {\it non-toral} if $[\gk,\gk] \not\subset \gh$. \label{Hsub}\end{definition}

Notice that if $\fH$ is connected, then the condition of $\Ad(\fH)$-invariance in the definition above is redundant. However, in the general case proper intermediate subalgebras which are not $\Ad(\fH)$-invariant can occur. \smallskip

Let us consider now an $\fH$-subalgebra $\gk\subset \gg$ and let $\fK^{\zero}$ be the only connected Lie subgroup of $\fG$ with Lie algebra $\Lie(\fK^{\zero})=\gk$. Of course, if $\fH$ is connected then $\fH \subset \fK^{\zero}$. However, in general it only holds that the identity component of $\fH$ stays in $\fH\cap \fK^{\zero}$ and there is no need for the whole subgroup $\fH$ to be contained in $\fK^{\zero}$. Anyway, we stress the following important fact.

\begin{prop} Let $\gk$ be an $\fH$-subalgebra of $\gg$ and $\fK^{\zero}$ be the only connected Lie subgroup of $\fG$ such that $\Lie(\fK^{\zero})=\gk$. Then, the subgroup $\fK$ generated by $\fH$ and $\fK^{\zero}$ is a Lie subgroup of $\fG$, not necessarily closed, with $\Lie(\fK)=\gk$. Moreover, $\fH$ is closed in $\fK$ and the quotient $\fK/\fH$ is connected. Finally, $\gk$ is toral if and only if $\ol{\fK}/\fH$ is a torus. \label{K} \end{prop}

\begin{proof} Since $\gk$ is $\Ad(\fH)$-invariant, it follows that $\fH$ normalizes $\fK^{\zero}$, i.e. $C(h)(\fK^{\zero})\subset \fK^{\zero}$ for any $h \in \fH$, where $C(\cdot)$ indicates the conjugation inside $\fG$. Therefore the subgroup $\fK \subset \fG$ generated by $\fH$ and $\fK^{\zero}$ coincides with the set $\fH\fK^{\zero}=\{hk : h\in \fH \, , \,\, k \in \fK^{\zero}\}$. Since $$\fH\fK^{\zero} \simeq (\fH{\times}\fK^{\zero}) /\, \fH\cap \fK^{\zero} \,\, ,$$ where $\fH\cap \fK^{\zero}$ acts freely on $\fH{\times}\fK^{\zero}$ on the right by $(h,k)\cdot h' \= (hh',(h')^{-1}k)$, it comes that $\fK$ is a Lie subgroup of $\fG$ which is closed if and only if $\fK^{\zero}$ is closed in $\fG$. Since the identity component of $\fH$ is contained in $\fK^{\zero}$, it follows that the identity component of $\fK$ coincides with $\fK^{\zero}$ and hence $\Lie(\fK)=\gk$.

We notice now that $\fK$ is Hausdorff and $\fH$ is compact, hence $\fH$ is necessarily closed in $\fK$. Moreover, by the Second Isomorphism Theorem we get $\fK/\fK^{\zero} \simeq \fH/(\fH\cap \fK^{\zero})$ and hence $\fK/\fH$ is connected.

Let us suppose now that $\gk$ is toral. We can also assume that $\fK^{\zero}$ is closed in $\fG$. Otherwise, one can just reply the same argument as below by replacing $\gk$ with its Malcev-closure $\bar{\gk}$ inside $\gg$. We notice that by the Second Isomorphism Theorem $\fK/\fH \simeq \fK^{\zero}/(\fH\cap \fK^{\zero})$ and that the subgroup $\fH\cap \fK^{\zero}$ is normal in $\fK^{\zero}$. To prove this last claim, firstly we observe that it is straightforward to show that the commutator $[\fK^{\zero},\fK^{\zero}]$ is connected. Therefore, since $[\gk,\gk] \subset \gh$ it holds that $[\fK^{\zero},\fK^{\zero}] \subset \fH\cap \fK^{\zero}$ and hence $C(k)(h)=[k,h]h \in \fH \cap \fK^{\zero}$ for any $k \in \fK^{\zero}$, $h \in \fH \cap \fK^{\zero}$. This actually proves that $\fK/\fH$ is a compact, connected Lie group. Finally, by using the fact that $[\gk,\gk] \subset \gh$, the Lie algebra $\gk$ splits as $$\gk=\gh\oplus\ga \,\, , \quad \text{ with $[\gh,\ga]=[\ga,\ga]=\{0\}$}$$ and therefore $\fK/\fH$ is a torus. On the other hand, it is easy to check that if $\ol{\fK}/\fH$ is a torus, then $[\gk,\gk]=[\bar{\gk},\bar{\gk}] \subset \gh$ and this completes the proof. \end{proof}

From now on, we will always associate to any $\fH$-subalgebra $\gk \subset \gg$ the Lie subgroup $\fK \subset \fG$ defined as in Proposition \ref{K}.  If $\fK$ is closed in $\fG$, then it gives rise to the homogeneous fibration $\fK/\fH \rar \fG/\fH \rar \fG/\fK$ whose standard fiber $\fK/\fH$, which is not almost-effective in general, is a torus if and only if $\gk$ is toral.

If $\fK$ is not closed in $\fG$, then there always exist a neighborhood $\eU_{\fK} \subset \fK$ of the unit in the manifold topology of $\fK$ and two neighborhoods $\eU_{\fH} \subset \fH$, $\eU_{\fG} \subset \fG$ of the unit such that $\eU_{\fH} \subset \eU_{\fK} \subset \eU_{\fG}$, the canonical immersions $\eU_{\fH} \hookrightarrow \eU_{\fK} \hookrightarrow \eU_{\fG}$ are embeddings and the {\it local factor spaces} $\eU_{\fK}/\eU_{\fH}$, $\eU_{\fG}/\eU_{\fH}$, $\eU_{\fG}/\eU_{\fK}$ are well defined. We refer to \cite{Go} for a self-contained treatment of the theory of local (Lie) groups and to \cite{Mos,Sp} for what concerns local factor spaces and locally homogeneous manifolds (see also \cite[Sec 6]{Ped}). Again, we get a fibration $\sU_{\fK}/\sU_{\fH} \rar \sU_{\fG}/\sU_{\fH} \rar \sU_{\fG}/\sU_{\fK}$ and the local factor spaces $\sU_{\fG}/\sU_{\fH}$ and $\sU_{\fK}/\sU_{\fH}$ are locally equivariantly diffeomorphic to the global homogeneous spaces $\fG/\fH$ and $\fK/\fH$, respectively (see also \cite[Note 1.2]{Sp}). Moreover, $\fK/\fH$ is a dense submanifold of $\ol{\fK}/\fH$, which is a torus if and only if $\gk$ is toral.

For the sake of simplicity, since we do not need an exact notation for local factor spaces, from now on we will always write $\fG/\fK$, either when $\fK$ is closed in $\fG$ or not. \smallskip

Any $\fH$-subalgebra $\gk$ determines an $\Ad(\fH)$-invariant $Q$-orthogonal decomposition \beq \gg=\underbrace{\gh + \gm_{\gk}}_{\gk} + \!\!\!\!\!\!\!\!\!\!\!\!\overbrace{\phantom{\gm_{\gk} +}\gm_{\gk}^{\perp}}^{\gm} \,\, , \quad \text{ with }\, [\gk,\gm_{\gk}^{\perp}] \subset \gm_{\gk}^{\perp} \,\, .\label{mk} \eeq Since $\gk$ is reductive, $\gk$ is toral if and only if $\gm_{\gk}$ lies in the center of $\gk$, i.e. $[\gh,\gm_{\gk}]=[\gm_{\gk},\gm_{\gk}]=\{0\}.$ If $\gk$ is not compact, i.e. if the subgroup $\fK$ is not closed in $\fG$, from the equality $[\ol{\gk},\ol{\gk}]=[\gk,\gk]$ we get a finer $\Ad(\fH)$-invariant $Q$-orthogonal decomposition \beq \gg=\underbrace{\gk + \ga_{\gk}}_{\ol{\gk}} + \!\!\!\!\!\!\!\!\!\!\!\!\overbrace{\phantom{\gm_{\gk} +}\gm_{\ol{\gk}}^{\perp}}^{\gm_{\gk}^{\perp}}=\gh+\gm_{\bar{\gk}}+\gm_{\bar{\gk}}^{\perp} \,\, , \quad \text{ with }\, [\gk,\ga_{\gk}]=[\ga_{\gk},\ga_{\gk}]=\{0\} \,\, .\label{mbark} \eeq We remark also that any submodule of $\gm$ is $\Ad(\ol{\fK})$-invariant if and only if is $\Ad(\fK)$-invariant.

Finally, if we suppose that the group $\fG$ is semisimple, given any not necessarily compact toral $\fH$-subalgebra $\gk$, the following result holds.

\begin{lemma} Let $\gk$ be an $\fH$-subalgebra of $\gg$. If $\fG$ is semisimple and $\gk$ is toral, then $\gk$ is faithfully represented by its adjoint action on $\gm_{\gk}^{\perp}$.  \label{lemmablue} \end{lemma}
\begin{proof} Since $\fG$ is compact and $\ol{\fK}$ is closed in $\fG$, the quotient $\fG/\ol{\fK}$ is a reductive homogeneous space. Let now $\fN$ be the maximal normal subgroup of $\fG$ contained in $\ol{\fK}$ and $\gn\=\Lie(\fN)$. We consider also the $Q$-orthogonal decomposition $\gn=\gn_1+\gn_2$, with $\gn_1\= \gh \cap \gn$. Since $\gn$ is an ideal of $\gg$ and $\gn \subset \bar{\gk}$, it follows that $[\gn,\gm_{\bar{\gk}}^{\perp}]=\{0\}$. Moreover, since $\gn_2 \subset \gm_{\bar{\gk}}$ and $\gk$ is toral, it holds that $[\gn_2,\gh]=[\gn_2,\gm_{\bar{\gk}}]=\{0\}$. But then $\gn_2 \subset \gz(\gg)=\{0\}$ and so $\gn=\gn_1 \subset \gh$. Being $\fG/\fH$ almost-effective by assumption, it follows that $\gn=\{0\}$ and so $\fG/\ol{\fK}$ is almost-effective. Hence, its isotropy representation is faithful (see e.g. \cite[Cor 6.15]{PoSp}). But then $$\{X \in \gk: [X,\gm_{\gk}^{\perp}]=\{0\}\} \subset \{X \in \bar{\gk}: [X,\gm_{\bar{\gk}}^{\perp}]=\{0\}\}=\{0\}$$ and so the claim follows. \end{proof}

\subsection{Submersion metrics and submersion directions} \hfill \par

As a standard reference for what concerns Riemannian submersion, we refer to \cite[Ch 9]{Bes}. We recall here the following

\begin{definition} Let $\gk \subset \gg$ be an $\fH$-subalgebra. An invariant metric $g \in \eM^{\fG}$ is called {\it $\gk$-submersion metric} if $g(\gm_{\gk},\gm_{\gk}^{\perp})=\{0\}$ and its restriction on $\gm_{\gk}^{\perp} \otimes \gm_{\gk}^{\perp}$ is $\Ad(\fK)$-invariant. The set of all $\gk$-submersion metrics is denoted by $\eM^{\fG}(\gk)$ and the set of unit volume $\gk$-submersion metrics is denoted by $\eM^{\fG}_1(\gk) \= \eM^{\fG}_1 \cap \eM^{\fG}(\gk)$. \end{definition}

This definition is due to the fact that, given an $\fH$-subalgebra $\gk$, any metric $g \in \eM^{\fG}(\gk)$ gives rise to a (locally) homogeneous Riemannian submersion \beq \fK/\fH \rar (\fG/\fH,g) \rar (\fG/\fK,g|_{\gm_{\gk}^{\perp} \otimes \gm_{\gk}^{\perp}}) \label{Riemsub} \,\, .\eeq Moreover, by means of the following lemma, the submersion \eqref{Riemsub} has totally geodesic fibers.

\begin{lemma} Let $\gk \subset \gg$ be an $\fH$-subalgebra, $\fK$ the corresponding Lie subgroup and $g \in \eM^{\fG}$. If $g(\gm_{\gk},\gm_{\gk}^{\perp}){=}\{0\}$ with respect to the decomposition \eqref{mk}, then $\fK/\fH$ is totally geodesic in $(\fG/\fH,g)$. \label{totg}\end{lemma}
\begin{proof} Let $X_1,X_2 \in \gm_{\gk}$ and $X_3 \in \gm_{\gk}^{\perp}$. Since by hypothesis $g(\gm_{\gk},\gm_{\gk}^{\perp})=\{0\}$, from \cite[Lemma 7.27]{Bes} we directly get that \begin{align*}
2g\big(\n^g_{X_1^*}X_2^*,X_3^*\big) &= g([X_1^*,X_2^*],X_3^*) +g([X_1^*,X_3^*],X_2^*) +g([X_2^*,X_3^*],X_1^*) \\
&= -g([X_1,X_2]_{\gm},X_3) +g([X_3,X_1]_{\gm},X_2) +g([X_3,X_2]_{\gm},X_1) \\
&= 0 \,\, , 
\end{align*} where we indicated with $X^*_x\=\frac{d}{dt}\exp(tX)\cdot x \big|_{t=0}$ the action vector field associated to $X \in \gg$, with $\n^g$ the Levi-Civita connection of $g$ and we used the fact that $[X,Y]^*=-[X^*,Y^*]$ for any $X,Y \in \gg$. This is equivalent of saying that the second fundamental form of $\fK/\fH$ in $(\fG/\fH,g)$ is identically zero, and so $\fK/\fH$ is totally geodesic. \end{proof}

Let now $\eM^{\fG}_1$ be the space of unit volume $\fG$-invariant metrics on $M=\fG/\fH$ and $\S \subset T_{Q_{\gm}}\eM^{\fG}_1$ the unit tangent sphere defined in \eqref{Sigma}. Fix $v \in \S$ and a good decomposition $\f$ for $v$. Let also $$\hat{v}_1<{\dots} <\hat{v}_{\ell_{v}}$$ be the distinct eigenvalues of $v$ ordered by size, and let $I_1^v(\f),{\dots},I_{\ell_v}^v(\f) \subset I=\{1,{\dots},\ell\}$ be the index sets defined by the condition \beq v_i =\hat{v}_s \iff i \in I_s^v(\f) \,\quad \text{ for every $s \in \{1,{\dots},\ell_v\}$, $i \in I$ } \,\, . \label{Iv} \eeq 

\begin{lem}[\cite{B\"o1}, Lemma 4.12 and Lemma 4.13] Let $v \in \S$ and let $\f$ be a good decomposition for $v$. Then $\ell_v >1$ and there exists a constant $c=c(\fG/\fH)>0$, which does not depend neither on $v$ nor $\f$, such that $\hat{v}_1<-c$ and $\hat{v}_{\ell_v}>c$. Furthermore, for any $1 \leq i,j,k \leq \ell_v$, the real number $[I^v_i(\f)I^v_j(\f)I^v_k(\f)]_{\f}$ does not depend on the choice of the good decomposition $\f$. \label{lemBo1} \end{lem}

From \eqref{scal} it follows that the scalar curvature along the geodesic $\g_v(t)$ is \beq \scal(\g_v(t)) = \frac12\sum_{i\in I}d_ib_ie^{-tv_i}-\frac14\sum_{i,j,k \in I}[ijk]_{\f}e^{t(v_i-v_j-v_k)} \,\, . \label{scalg}\eeq We recall now the following definition, firstly introduced by B{\"o}hm.

\begin{definition}[\cite{B\"o1}, Def 5.11] Let $\eS^{\S}$ denote the set of all $v \in \S$ with the following property: if $\f$ is any good decomposition for $v$, then for all $(i,j,k) \in I^3$ it holds that \beq [ijk]_{\f}>0 \,\,\Longrightarrow\,\, v_i-v_j-v_k +\hat{v}_1\leq 0 \,\, . \label{subdir} \eeq Any element $v \in \eS^{\S}$ is called {\it submersion direction}. \end{definition}

Notice that \eqref{subdir} does not depend on the choice of the good decomposition $\f$ for $v$. Moreover, submersion directions (or {\it non-negative directions}, as originally named by B{\"o}hm) have the following remarkable property, which comes directly from \eqref{subdir}.

\begin{prop}[\cite{B\"o1}, Lemma 5.16] Let $v \in \eS^{\S}$ and let $\f$ be a good decomposition for $v$. Then \beq [I^v_1(\f)I^v_{j_1}(\f)I^v_{j_2}(\f)]_{\f}=0 \quad \text{ for any }\, 1 \leq j_1 < j_2 \leq \ell_v \,\, . \label{crucial}\eeq In particular, $\gk_1\=\gh + \gm_{I^v_1(\f)}$ is an $\fH$-subalgebra. \label{propBo1} \end{prop}

This last proposition gives rise to a partition of the set $\eS^{\S}$ into the sets of {\it $\gk_1$-submersion directions}, which are defined by \beq \eS^{\S}(\gk_1) \= \{v \in \eS^{\S} : \gm_{I^v_1(\f)}=\gm_{\gk_1} \,\text{ for any good decomposition } \f \text{ for } v \} \,\, , \label{k-sub} \eeq for any $\fH$-subalgebra $\gk_1 \subset \gg$. As a direct generalization of \eqref{k-sub}, we are going to introduce a descending chains of subsets of $\eS^{\S}$, which will play a role in the next section. First, we define {\it flag of $\fH$-subalgebras} any ordered set $\z \=(\gk_1,{\dots},\gk_p)$ of $\fH$-subalgebras of $\gg$ such that $\gk_1 \subsetneq {\dots} \subsetneq \gk_p$. The {\it lenght of $\z$} is the cardinality $|\z|=p$. Notice that, by Proposition \ref{K}, any flag of $\fH$-subalgebras determines univocally a finite sequence of intermediate Lie subgroups $\fH \subsetneq \fK_1\subsetneq {\dots} \subsetneq \fK_p \subsetneq \fG$.

\begin{definition} Let $\z\= (\gk_1,{\dots},\gk_p)$ be a flag of $\fH$-subalgebras. A unit tangent vector $v \in \S$ is called {\it $\z$-submersion direction} if it satisfies the following conditions for any good decomposition $\f$ of $v$: \begin{itemize}
\item[i)] $\gk_1=\gh+\gm_{I^v_1(\f)} \, , \,\, \gk_2=\gk_1+\gm_{I^v_2(\f)} \, , {\dots}, \,\, \gk_p=\gk_{p-1}+\gm_{I^v_p(\f)}$ ;
\item[ii)] for any $1 \leq q \leq p$, for any $(i,j,k) \in \{q,{\dots},\ell_v\}^3$ it holds $$[I_i^v(\f)I_j^v(\f)I_k^v(\f)]_{\f}>0 \,\,\,\Longrightarrow\,\,\, \hat{v}_i-\hat{v}_j-\hat{v}_k+\hat{v}_q\leq 0 \,\, .$$
\end{itemize} The set of all $\z$-submersion directions is denoted by $\eS^{\S}(\z)$ or $\eS^{\S}(\gk_1,{\dots},\gk_p)$, equivalently. \label{zsub}\end{definition}

Given a flag of $\fH$-subalgebras $\z\= (\gk_1,{\dots},\gk_p)$, it follows from the very definition that $$\eS^{\S}(\z)=\eS^{\S}(\gk_1,{\dots},\gk_p) \subset \eS^{\S}(\gk_1,{\dots},\gk_{p-1}) \subset {\dots} \subset \eS^{\S}(\gk_1,\gk_2) \subset \eS^{\S}(\gk_1) \,\, .$$ Furthermore, the set $\eS^{\S}(\z)$ of $\z$-submersion directions is related with the notion of submersion type metrics by the following

\begin{prop} Let $\z=(\gk_1,{\dots},\gk_p)$ be a flag of $\fH$-subalgebras. Then, it holds that \beq \eS^{\S}(\z) \subset \eS^{\S} \cap T_{Q_{\gm}}\eM^{\fG}_1(\gk_q) \quad \text{ for any }\, 1\leq q \leq p \,\, , \label{expW} \eeq i.e. $\g_v(t) \in \eM^{\fG}_1(\gk_q)$ for any $v \in \eS^{\S}(\z)$, for any $t >0$, for any $1\leq q \leq p$. \label{propflag} \end{prop}
\begin{proof} Let $v \in \eS^{\S}(\z)$ and $\f$ be a good decomposition for $v$. Fix $1\leq q \leq p$. We have to show that the submodule $\gm_{I^v_i(\f)}$ is $\Ad(\fK_q)$-invariant for any $q \leq i \leq \ell_v$. Since every submodule $\gm_{I^v_i(\f)}$ is $\Ad(\fH)$-invariant, it follows from the very definition of $\fK_q$ (see Proposition \ref{K}) that it is sufficient to show that $\gm_{I^v_i(\f)}$ is $\ad(\gk_q)$-invariant for any $q \leq i \leq \ell_v$. We already know from \eqref{mk} that $[\gk_q,\gm_{\gk_q}^{\perp}]\subset \gm_{\gk_q}^{\perp}$. From condition (ii) in Definition \ref{zsub}, we get $$[I_q^v(\f)I_{j_1}^v(\f)I_{j_2}^v(\f)]=0 \quad \text{ for any $q \leq j_1 < j_2 \leq \ell_v$ }\, .$$ In particular, $Q\big([\gm_{\gk_q},\gm_{I^v_i(\f)}],\gm_{I^v_j(\f)}\big)=0$ for any $q < i,j \leq \ell_v$, $i\neq j$. So, we can conclude that $[\gm_{\gk_q},\gm_{I_i^v(\f)}] \subset \gm_{I_i^v(\f)}$ for any $q < i \leq \ell_v$. \end{proof}

By means of Proposition \ref{propflag}, the following geometric interpretation for the set of submersion directions arises. Given an element $v \in \eS^{\S}(\gk_1)$, moving along the geodesic $\g_v(t)$ is equivalent to shrinking the fibers of the (locally) homogeneous Riemannian submersion associated to $\gk_1$ as in \ref{Riemsub} and to rescaling the base space, while the volume is keeped fixed. \smallskip

The set $\eS^{\S} \subset \S$ of submersion directions has originally raised from the study of the scalar curvature functional $\scal: \eM^{\fG}_1 \rar \bR$, aimed to get results of existence and non-existence for homogeneous Einstein metrics (see e.g. \cite{WZ, B\"o1}). It turns out that it plays a crucial role in studying the asymptotic behavior of the curvature tensor along geodesic rays $\g_v$. More concretely

\begin{theo} Let $v \in \S$ and $\g_v$ the corresponding geodesic ray in $\sM^G_1$. \begin{itemize}
\item[a)] \cite[Thm 5.18]{B\"o1} If $v \in \S \setminus \eS^{\S}$, then $\lim_{t \rar +\infty}\scal(\g_v(t)) \rar -\infty$.
\item[b)] If $v \in \eS^{\S}(\gk_1)$ for a non-toral $\fH$-subalgebra $\gk_1 \subset \gg$, then $\lim_{t \rar +\infty}\big|\Ric(\g_v(t))\big|_{\g_v(t)} \rar +\infty$.
\end{itemize} \label{propgeo} \end{theo}

\begin{proof} Fix $v \in \S$ and a good decomposition $\f$ for $v$. If $v \in \S \setminus \eS^{\S}$, then there exists $\e>0$ and a triple $(i_{\zero},j_{\zero},k_{\zero}) \in I^3$ such that $[i_{\zero}j_{\zero}
k_{\zero}]_{\f}>\e$ and $v_{i_{\zero}}-v_{j_{\zero}}-v_{k_{\zero}}+\hat{v}_1>\e$. Since $\hat{v}_1< 0$ by Lemma \ref{lemBo1}, from \eqref{scalg} we get $$\scal(\g_v(t)) < \frac12\big(b_{\fG/\fH}-\e e^{t\e}\big)e^{-t\hat{v}_1} \rar -\infty$$ and this completes the proof of the first claim.

Let now $\gk_1$ be a non-toral $\fH$-subalgebra of $\gg$ and suppose that $v \in \eS^{\S}(\gk_1)$. Then, if $i \in I_1^v(\f)$, for any $j,k \in I$ it follows from \eqref{crucial} that \beq \begin{gathered} \text{ $[ijk]_{\f}\big(1-e^{t(v_k-v_j)}\big)=0$ for any $t>0$ } \, , \\ \text{$[ijk]_{\f}>0$ only if $j,k \in I_s^v(\f)$ for some $1 \leq s \leq \ell_v$ } \, . \label{condforRic} \end{gathered}\eeq So, for any $i \in I_1^v(\f)$, from \eqref{ric} we get \begin{align*}
\ric_i(\g_v(t)) &\overset{{\color{white} (2.13)}}{=} \frac{b_i}{2}e^{-tv_i}-\frac1{2d_i}\sum_{j,k \in I}[ijk]_{\f}e^{t(v_k-v_i-v_j)}+\frac1{4d_i}\sum_{j,k \in I}[ijk]_{\f}e^{t(v_i-v_j-v_k)} \\
&\overset{\eqref{dbc}}{=} \Big(c_i+\frac1{2d_i}\sum_{j,k \in I}[ijk]_{\f}\Big)e^{-t\hat{v}_1}-\frac1{2d_i}e^{-t\hat{v}_1}\sum_{j,k \in I}[ijk]_{\f}e^{t(v_k-v_j)}+\frac1{4d_i}e^{t\hat{v}_1}\sum_{j,k \in I}[ijk]_{\f}e^{-t(v_j+v_k)} \\
&\,\overset{\eqref{condforRic}}{=}\,\, c_ie^{-t\hat{v}_1}+\frac1{4d_i}e^{t\hat{v}_1}\sum_{\substack{j,k \in I_s^v(\f) \\ 1\leq s \leq \ell_v}}[ijk]_{\f}e^{-2t\hat{v}_s} \\
&\overset{{\color{white} (2.13)}}{=} \frac1{2d_i}\Big(2d_ic_i+\frac12\sum_{j,k \in I_1^v(\f)}[ijk]_{\f}\Big)e^{-t\hat{v}_1}+\frac1{4d_i}\sum_{\substack{j,k \in I_s^v(\f) \\ 2\leq s \leq \ell_v}}[ijk]_{\f}e^{-t(2\hat{v}_s-\hat{v}_1)} \,\, .
\end{align*} Since $\gk_1$ is non toral, there exists $i_{\zero} \in I^v_1(\f)$ such that $$2d_{i_{\zero}}c_{i_{\zero}}+\frac12\sum_{j,k \in I_1^v(\f)}[i_{\zero}jk]_{\f}>0$$ and so the second claim follows. \end{proof}

\begin{remark} To prove the second claim, it is possible to argue also like this. Let $v \in \eS^{\S}(\gk_1)$ for a given non-toral $\fH$-subalgebra $\gk_1$ and $\f \in \eF^{\fG}$ be a good decomposition for $v$. Since $\g_v(t)|_{\fK_1/\fH}=e^{t\hat{v}_1}Q_{I^v_1(\f)}$ and $\hat{v}_1<0$, it follows that the intrinsic sectional curvature of $\fK_1/\fH$ blows up as $t\rar +\infty$. Moreover, from Lemma \ref{totg} and Proposition \ref{propflag}, we know that $\fK_1/\fH$ is totally geodesic in $(\fG/\fH,\g_v(t))$ for any $t>0$ and so also its extrinsic sectional curvature blows up. Then, claim (b) follows directly from \cite[Thm 4]{BLS}. \end{remark}

As a consequence of Theorem \ref{propgeo}, the only way of {\it reaching the boundary of the space $\eM^{\fG}_1$}, moving along a geodesic $\g_v$ while keeping the curvature bounded, is to choose $v \in \eS^{\S}(\gk_1)$ for some toral $\fH$-subalgebra $\gk_1 \subset \gg$. By the way, we stress the fact that this last condition is far form being sufficient.

\begin{example}[Berger spheres] Let $M=\fG=\fSU(2)$. Consider the $\Ad(\fSU(2))$-invariant inner product $Q(A_1,A_2)\={-}\frac12\Tr(A_1{\cdot}A_2)$ on $\su(2)$, the standard $Q$-orthonormal basis $\eB=(X_1,X_2,X_3)$ such that $$[X_1,X_2]=-2X_3 \,\, , \quad [X_2,X_3]=-2X_1 \,\, , \quad [X_3,X_1]=-2X_2$$ and set $\gk:=\vspan(X_1)$. By means of \eqref{Sigma} and \eqref{crucial}, it is easy to check that $\eS^{\S}(\gk)=\{\bar{v}\}$, where the tangent direction $\bar{v}$ is given, with respect to the basis $\eB$, by $$\bar{v}=\left({\scalefont{0.8}\begin{array}{ccc} {-}\frac{\sqrt6}3 \!&\! \!&\! \\ \!&\! \frac{\sqrt6}6 \!&\! \\ \!&\! \!&\! \frac{\sqrt6}6 \end{array}}\right) \,\, .$$ Let us indicate now with $\big(X_1(t)\=e^{\frac{\sqrt6}6t}X_1, X_2(t)\=e^{-\frac{\sqrt6}{12}t}X_2, X_3(t)\=e^{-\frac{\sqrt6}{12}t}X_3\big)$ the $\g_{\bar{v}}(t)$-orthonormal basis for $\su(2)$ obtained by normalizing $\eB$. Then, one can directly check that the curvature tensor $$\Rm(\g_{\bar{v}}(t)): \su(2) \wedge \su(2) \rar \su(2) \wedge \su(2)$$ is diagonal and explicitly given by $$\begin{gathered} \Rm(\g_{\bar{v}}(t))(X_1(t){\wedge}X_2(t))=e^{-\frac23\sqrt6t}X_1(t){\wedge}X_2(t) \,\, , \\ \Rm(\g_{\bar{v}}(t))(X_1(t){\wedge}X_3(t))=e^{-\frac23\sqrt6t}X_1(t){\wedge}X_3(t) \,\, , \\ \Rm(\g_{\bar{v}}(t))(X_2(t){\wedge}X_3(t))=\Big(4e^{-\frac{\sqrt{6}}6t}-3e^{-\frac23\sqrt6t}\Big)X_2(t){\wedge}X_3(t) \,\, .\end{gathered}$$ Hence, we conclude that $\lim_{t\rar+\infty}\big|\Rm(\g_{\bar{v}}(t))\big|_{\g_{\bar{v}}(t)}=0$. Notice that $\g_{\bar{v}}(t)$ comes from the canonical variation of the round metric on $S^3=\fSU(2)$ with respect to the Hopf fibration $S^1 \rar S^3 \rar S^2=\fSU(2)/S^1$ (see \cite[p. 252]{Bes}). When endowed with any such a metric, the $3$-sphere is called a {\it Berger sphere}. \label{Berger} \end{example}

\smallskip

\section{Proofs of Theorem \ref{MAIN} and Theorem \ref{MAIN2}} \label{section4} \setcounter{equation} 0

\subsection{Main results} \hfill \par

Let us consider a sequence $(g^{(n)}) \subset \eM^{\fG}_1$. Then, for every $n \in \bN$ there exist $v^{(n)} \in \S$ and $t^{(n)} >0$, univocally determined, such that $g^{(n)}=\g_{v^{(n)}}(t^{(n)})$.  Since $\S$ is compact, there exist a sequence $(n_i) \subset \bN$ and a direction $v^{(\infty)} \in \S$ such that $v^{(n_i)} \rar v^{(\infty)}$. For the sake of simplicity, in this section we will assume that the whole sequence $(v^{(n)})$ converges to some $v^{(\infty)} \in \S$, which we call {\it limit direction of $(g^{(n)})$}. We also say that $(g^{(n)})$ is {\it divergent} if $t^{(n)} \rar +\infty$. \smallskip

For any $n \in \bN$ we choose a good decomposition $\f^{(n)}=(\gm_1^{(n)},{\dots},\gm_{\ell}^{(n)})$ of $\gm$ for $v^{(n)}$, so that \beq g^{(n)}=\l_1^{(n)}Q_{\gm_1^{(n)}}+{\dots}+\l_{\ell}^{(n)}Q_{\gm_{\ell}^{(n)}} \,\, , \quad \text{ with } \,\,\l_i^{(n)}\=e^{t^{(n)}v_i^{(n)}} \,\, . \label{diag(n)} \eeq Since $v^{(n)}\rar v^{(\infty)}$, we can suppose that the sequence $(\f^{(n)}) \subset \sF^{\fG}$ converges as $n \rar +\infty$ to a good decomposition $\f^{(\infty)}=(\gm_1^{(\infty)},{\dots},\gm_{\ell}^{(\infty)})$ for the limit direction $v^{(\infty)}$ of $(g^{(n)})$. For simplicity of notation, since we do not need to specify the particular choice of $\f^{(n)}$ and $\f^{(\infty)}$, we will write $[ijk]^{(n)}$ and $[ijk]^{(\infty)}$ instead of $[ijk]_{\f^{(n)}}$ and $[ijk]_{\f^{(\infty)}}$, respectively. Being the map $\f \mapsto [ijk]_{\f}$ continuous, it holds that $[ijk]^{(n)} \rar [ijk]^{(\infty)}$ as $n \rar +\infty$. Furthermore, the coefficients introduced in \eqref{b_i} and \eqref{c_i} will be indicated by $b_i^{(n)}$, $c_i^{(n)}$ when they refer to the decomposition $\f^{(n)}$ and by $b_i^{(\infty)}$, $c_i^{(\infty)}$ when they refer to the decomposition $\f^{(\infty)}$, respectively. Again, it holds that $b_i^{(n)} \rar b_i^{(\infty)}$ and $c_i^{(n)} \rar c_i^{(\infty)}$ as $n \rar +\infty$.

From now on, up to passing to a subsequence we will always assume that the decompositions $\f^{(n)}$ are ordered in such a way that \beq v_1^{(n)} \leq v_2^{(n)}\leq {\dots} \leq v_{\ell}^{(n)} \quad \text{ for any $n \in \bN$ } \,\, . \label{goodorder} \eeq For simplicity of notation, we set $I\=\{1,{\dots},\ell\}$, $I_s^{(\infty)} \= I_s^{v^{(\infty)}}(\f^{(\infty)})$ for any $1 \leq s \leq \ell_{v^{(\infty)}}$ and we define the map $r:\{0,{\dots},\ell_{v^{(\infty)}}\} \rar \{0,{\dots},\ell\}$ by imposing the conditions \beq r(0)\=0 \,\, , \quad I_s^{(\infty)}=\{r(s-1)+1,{\dots},r(s)\} \quad \text{ for any } 1 \leq s \leq \ell_{v^{(\infty)}} \,\, . \label{r} \eeq Moreover, we set $I^{(\infty)}_{\geq q} \= \bigcup_{s=q}^{\ell_{v^{(\infty)}}}I^{(\infty)}_s$. Let us fix for each $n \in \bN$ a $Q_{\gm}$-orthonormal $\f^{(n)}$-adapted basis $\big(e^{(n)}_{\a}\big)$ for $\gm$. Since $v^{(n)} \rar v^{(\infty)}$ we can suppose that there exists a $Q_{\gm}$-orthonormal $\f^{(\infty)}$-adapted basis $\big(e_{\a}^{(\infty)}\big)$ for $\gm$ such that $e^{(n)}_{\a} \rar e_{\a}^{(\infty)}$ as $n\rar +\infty$. For the sake of shortness we set \begin{align} &\sec_i\!\big(g^{(n)}\big)\=\sum_{e_{\a}^{(n)},e_{\a'}^{(n)}\in \gm_i^{(n)}}\sec(g^{(n)})(e_{\a}^{(n)}{\wedge}e_{\a'}^{(n)}) \quad \text{ for any $i \in I$ }\,\, , \label{seci}\\ &\sec_{ij}\!\big(g^{(n)}\big)\=\sum_{\substack{e_{\a}^{(n)} \in \gm_i^{(n)}\\e_{\b}^{(n)} \in \gm_j^{(n)}}} \sec(g^{(n)})(e_{\a}^{(n)}{\wedge}e_{\b}^{(n)}) \quad \text{ for any $i, j \in I$, $i <j$ } \,\, . \label{secij}\end{align} From \eqref{coef1}, \eqref{sec1} and \eqref{dcbra} we obtain \begin{align}
\sec_i\!\big(g^{(n)}\big) &= \!\sum_{e_{\a}^{(n)},e_{\a'}^{(n)}\in \gm_i^{(n)}} \!\!\bigg\{\big|[e_{\a}^{(n)},e_{\a'}^{(n)}]_{\gh}\big|_Q^2+\frac14\big|[e_{\a}^{(n)},e_{\a'}^{(n)}]_{\gm_i^{(n)}}\big|_Q^2+\sum_{k\in I\setminus\{i\}}\big|[e_{\a}^{(n)},e_{\a'}^{(n)}]_{\gm_k^{(n)}}\big|_Q^2- \nonumber \\ 
&\phantom{aaaaaaaaaaaaaaaaaaaaaaaaaaaaaaaaaaaaaaaaaaaa}-\frac34\sum_{k\in I\setminus\{i\}}\big|[e_{\a}^{(n)},e_{\a'}^{(n)}]_{\gm_k^{(n)}}\big|_Q^2\frac{\l_k^{(n)}}{\l_i^{(n)}} \bigg\}\frac1{\l_i^{(n)}} \nonumber \\
&= \bigg(d_ic_i^{(n)}+\frac14[iii]^{(n)}+\sum_{k\in I\setminus\{i\}}[iik]^{(n)}-\frac34\sum_{k\in I\setminus\{i\}}[iik]^{(n)}\frac{\l_k^{(n)}}{\l_i^{(n)}}\bigg)\frac1{\l_i^{(n)}} \,\, . \label{seci'}
\end{align} Moreover, from \eqref{coef1} and \eqref{sec2} we get \begin{align}
\sec_{ij}\!\big(g^{(n)}\big) &= \!\sum_{\substack{e_{\a}^{(n)} \in \gm_i^{(n)}\\e_{\b}^{(n)} \in \gm_j^{(n)}}}\!\!\Bigg\{\sum_{k \in I}\big|[e_{\a}^{(n)},e_{\b}^{(n)}]_{\gm_k^{(n)}}\big|_Q^2\frac{\big(\l_i^{(n)}\big)^2+\big(\l_j^{(n)}-\l_k^{(n)}\big)\big(-2\l_i^{(n)}+\l_j^{(n)}+3\l_k^{(n)}\big)}{4\l_i^{(n)}\l_j^{(n)}\l_k^{(n)}}\Bigg\} \nonumber \\
&= \frac14\sum_{k \in I}[ijk]^{(n)}\frac{\l_i^{(n)}}{\l_j^{(n)}\l_k^{(n)}}+\frac14\sum_{k \in I}[ijk]^{(n)}\Big(\frac{\l_j^{(n)}}{\l_k^{(n)}}-1\Big)\Big(-2\frac{\l_i^{(n)}}{\l_j^{(n)}}+1+3\frac{\l_k^{(n)}}{\l_j^{(n)}}\Big)\frac1{\l_i^{(n)}} \,\, . \label{secij'}
\end{align}

Up to passing to a subsequence we assume that each coefficient $\l_i^{(n)}$ is monotonic. Moreover, we introduce the following notation \beq p_{ij}^{(n)} \= \frac{\l_i^{(n)}}{\l_j^{(n)}} \label{pij} \eeq and, up to passing to a further subsequence, we assume that the limits $p_{ij}^{(\infty)} \= \lim_{n}p_{ij}^{(n)} \in [0,+\infty]$ do exist. Moreover, we define \beq a_{ijk}^{(n)} \= [ijk]^{(n)}\big(p_{jk}^{(n)}-1\big)\big(-2p_{ij}^{(n)}+1+3p_{kj}^{(n)}\big) \label{aijk} \eeq and we set $a_{ijk}^{(\infty)} \= \lim_{n}a_{ijk}^{(n)} \in \bR \cup \{\pm\infty\}$ whenever it exists. \smallskip

The next theorem is an intermediate result, which will be crucial in the proof of Theorem \ref{main2}. Nonetheless, we stress that it would be enough for proving Theorem \ref{MAIN}.

\begin{theo} Let us assume that $(g^{(n)}) \subset \eM^{\fG}_1$ is divergent and has bounded curvature. Then, $v^{(\infty)} \in \eS^{\S}(\gk_1)$ for some toral $\fH$-subalgebra $\gk_1$. Moreover, the following necessary conditions hold. \begin{itemize}
\item[\rm{A)}] For any $i\leq j \leq k$ such that $i\in I_1^{(\infty)}$, we have $$[ijk]^{(\infty)}=0 \quad \Longrightarrow \quad \lim_{n\rar+\infty}[ijk]^{(n)}p_{kj}^{(n)}=0 \,\, .$$
\item[\rm{B)}] For any $j,k \in I$ we have $$[I_1^{(\infty)}jk]^{(\infty)}>0 \quad \Longrightarrow \quad p_{kj}^{(\infty)}=1 \,\, .$$
\end{itemize} \label{preltheo} \end{theo}
\begin{proof} From \eqref{scalg} it follows that \begin{align*}
\scal(g^{(n)}) &=\frac12\sum_{i\in I}d_ib_i^{(n)}e^{-t^{(n)}v_i^{(n)}}-\frac14\sum_{i,j,k \in I}[ijk]^{(n)}e^{t^{(n)}(v^{(n)}_i-v^{(n)}_j-v^{(n)}_k)} \\
&\leq \frac14\Big(2b_{\fG/\fH}-\sum_{i,j,k \in I}[ijk]^{(n)}e^{t^{(n)}(v^{(n)}_i-v^{(n)}_j-v^{(n)}_k+v^{(n)}_1)}\Big)e^{-t^{(n)}v^{(n)}_1} \,\, ,
\end{align*} where $b_{\fG/\fH}$ is defined in \eqref{bG/H}. Since by assumption $\scal(g^{(n)})$ is bounded from below, there exists a constant $C>0$ such that \beq \sum_{i,j,k \in I}[ijk]^{(n)}e^{t^{(n)}(v^{(n)}_i-v^{(n)}_j-v^{(n)}_k+v^{(n)}_1)} \leq C \,\, \text{ for any $n \in \bN$ } \,\, . \label{boundscal} \eeq We observe also that if $v^{(\infty)} \in \S \setminus \eS^{\S}$, then \eqref{boundscal} is never satisfied. In fact, in that case we can fix $\e>0$ and a triple $(i_{\zero},j_{\zero},k_{\zero}) \in I^3$ such that $[i_{\zero}j_{\zero}k_{\zero}]^{(n)}>\e$ and $v_{i_{\zero}}^{(n)}-v_{j_{\zero}}^{(n)}-v_{k_{\zero}}^{(n)}+v_1^{(n)}>\e$, so that $$[i_{\zero}j_{\zero}k_{\zero}]^{(n)}e^{t^{(n)}(v^{(n)}_{i_{\zero}}-v^{(n)}_{j_{\zero}}-v^{(n)}_{k_{\zero}}+v^{(n)}_1)} > \e e^{t^{(n)}\e} \rar +\infty \,\, .$$ Then, it holds that $v^{(\infty)} \in \eS^{\S}(\gk_1)$ with $\gk_1 \= \gh + \gm^{(\infty)}_{I^{(\infty)}_1}$ (see Proposition \ref{propBo1}). Since by assumption the sectional curvature is bounded, using \eqref{seci'} and \eqref{secij'}, for any $i,j \in I$ such that $i \in I^{(\infty)}_1$, $i<j$ we get \begin{gather} \sec_i(g^{(n)})\cdot \l_i^{(n)} = d_ic_i^{(n)}+\frac14[iii]^{(n)}+\sum_{k\in I\setminus\{i\}}[iik]^{(n)}-\frac34\sum_{k\in I\setminus\{i\}}[iik]^{(n)}p_{ki}^{(n)} \,\longrightarrow\, 0 \,\, , \label{limseci} \\ \sec_{ij}(g^{(n)})\cdot 4\l_i^{(n)} = \sum_{k \in I}\big([ijk]^{(n)}p_{ik}^{(n)}p_{ij}^{(n)}+a_{ijk}^{(n)}\big) \,\longrightarrow\, 0 \label{limsecij} \end{gather} as $n \rar +\infty$, where $\sec_i(g^{(n)})$, $\sec_{ij}(g^{(n)})$ were defined in \eqref{seci}, \eqref{secij}, respectively, and the coefficients $p_{ij}^{(n)}$, $a_{ijk}^{(n)}$ were introduced in \eqref{pij}, \eqref{aijk}, respectively.

{\it Step 1.} We are going to apply \eqref{limsecij} by restricting ourselves to the case $j \in I^{(\infty)}_{\geq2}$. At first we notice that, since $i \leq r(1) < j$, for any $k \in I$ we have $$2v_i^{(n)}-v_k^{(n)}-v_j^{(n)} \,\,\longrightarrow\,\, 2\hat{v}_1^{(\infty)}-v_k^{(\infty)}-v_j^{(\infty)} \leq \hat{v}_1^{(\infty)}-\hat{v}_2^{(\infty)}<0 \,\, ,$$ where $\hat{v}_i^{(\infty)}$ are the distinct eigenvalues of $v^{(\infty)}$ ordered by size, and so \beq \lim_{n\rar+\infty}[ijk]^{(n)}p_{ik}^{(n)}p_{ij}^{(n)}=0 \quad \text{ for any } i,j,k \in I \text{ such that } i \in I^{(\infty)}_1 , \,  j \in I^{(\infty)}_{\geq2} \,\, . \label{st1inf} \eeq Therefore, from \eqref{limsecij} and \eqref{st1inf} we obtain for any fixed $j \in I^{(\infty)}_{\geq2}$ \beq \lim_{n\rar+\infty} \bigg\{ \sum_{k\in I} a_{ijk}^{(n)}\bigg\}=0\,\, , \quad \text{ for any } i \in I^{(\infty)}_1 \,\, .\label{starj} \tag{$\star_j$} \eeq Notice that, under the assumption $i \in I_1^{(\infty)}$ and $j \in I^{(\infty)}_{\geq2}$, it comes $p_{ij}^{(\infty)}=0$ and so from \eqref{aijk} we directly get the following implications: \beq \begin{aligned}
p_{jk}^{(\infty)}=+\infty \quad &\Longrightarrow \quad a_{ijk}^{(n)} \sim [ijk]^{(n)}p_{jk}^{(n)} \geq 0 \\
p_{jk}^{(\infty)} \in [1,+\infty) \quad &\Longrightarrow \quad a_{ijk}^{(\infty)} = [ijk]^{(\infty)}\big(p_{jk}^{(\infty)}-1\big)\big(1+3p_{kj}^{(\infty)}\big) \geq 0 \\
p_{jk}^{(\infty)} \in (0,1) \quad &\Longrightarrow \quad a_{ijk}^{(\infty)} = -[ijk]^{(\infty)}\big(1-p_{jk}^{(\infty)}\big)\big(1+3p_{kj}^{(\infty)}\big) \leq 0 \quad\quad . \\
p_{jk}^{(\infty)}=0 \quad &\Longrightarrow \quad a_{ijk}^{(n)} \sim -3[ijk]^{(n)}p_{kj}^{(n)} \leq 0
\end{aligned} \label{asymp} \eeq

For any $q \in \{0,1,{\dots},\ell-r(1)-1\}$, we set $j=\ell-q$ and we consider the following claim, which we denote by $P(q)$: the limit $a_{i(\ell-q)k}^{(\infty)}$ exists for any $i \in I_1^{(\infty)}$, $k \in I$ and $a_{i(\ell-q)k}^{(\infty)}=0$.

First, we consider the case $q=0$, i.e. $j=\ell$. From \eqref{goodorder}, we directly get that $p_{\ell k}^{(\infty)} \in [1,+\infty]$. But then, by means of \eqref{asymp} and $(\star_{\ell})$, it follows that $P(0)$ holds.

Let us fix now $0 \leq q \leq \ell-r(1)-2$ and assume that $P(q')$ holds for any $0 \leq q' \leq q$. In particular, this means that $a_{i(\ell-q')k}^{(\infty)}=0$ for any $i\in I_1^{(\infty)}$, $k \in I$ and hence for any $1 \leq q' \leq q$ we have \beq \left\{\begin{array}{ll}
\lim_{n\rar+\infty}[i(\ell{-}q')k]^{(n)}p_{(\ell{-}q')k}^{(n)}=0 & \text{ for any } i \in I^{(\infty)}_1 \, , \,\, k \in I\setminus\{\ell{-}q'\} \text{ such that } [i(\ell{-}q')k]^{(\infty)}=0 \\
p_{(\ell{-}q')k}^{(\infty)}=1 & \text{ for any $k\in I$ such that $[I_1^{(\infty)}(\ell{-}q')k]^{(\infty)}>0$ } \label{jq'}
\end{array}\right. \,\, .\eeq Then, for any $i \in I_1^{(\infty)}$, $k \in I$ we obtain: \begin{itemize}
\item[$\bcdot$] if $p_{(\ell{-}q{-}1)k}^{(\infty)} \in [1,+\infty]$, then, by \eqref{asymp}, we directly get that $a_{i(\ell{-}q{-}1)k}^{(n)}$ is definitely non negative;
\item[$\bcdot$] if $p_{(\ell{-}q{-}1)k}^{(\infty)} \in [0,1)$, then, by \eqref{goodorder}, it follows that there exists $1\leq q' \leq q$ such that $k=\ell{-}q'$ and so \eqref{asymp}, \eqref{jq'} imply that the limit $a_{i(\ell{-}q{-}1)k}^{(\infty)}$ exists and $a_{i(\ell{-}q{-}1)k}^{(\infty)}=0$.
\end{itemize} By means of $(\star_{\ell{-}q{-}1})$, this actually proves that $P(q+1)$ holds. Hence, we proved by induction that $P(q)$ holds for any $0 \leq q \leq \ell-r(1)-1$. In particular this means that $$a_{ijk}^{(\infty)}=0 \quad \text{ for any } i\in I^{(\infty)}_1 \, , \,\, j \in I^{(\infty)}_{\geq2} \, , \,\, k \in I$$ and hence the following two conditions must hold: \begin{align}
\text{ $i\in I^{(\infty)}_1$ , \, $j \in I^{(\infty)}_{\geq2}$ , \, $k \in I$ \, and \, $[ijk]^{(\infty)}=0$ } \quad &\Longrightarrow \quad \lim_{n\rar+\infty}[ijk]^{(n)}p_{jk}^{(n)}=0 \,\, ,\label{bound1} \\
\text{ $j,k \in I^{(\infty)}_{\geq2}$ \, and \, $[I_1^{(\infty)}jk]^{(\infty)}>0$ }
\quad &\Longrightarrow \quad p_{jk}^{(\infty)}=1 \,\, . \label{bound2} \end{align}

{\it Step 2.} We are going to apply \eqref{limsecij} by restricting ourselves to the case $j \in I^{(\infty)}_1$. For the sake of clarity, we set $i_1 \= i$ and $i_2 \= j$. At first we notice that, since $i_1 < i_2 \leq r(1)$, for any $k \in I^{(\infty)}_{\geq2}$ \beq a_{i_1i_2k}^{(n)} \,\sim\, -3[i_1i_2k]^{(n)}p_{ki_2}^{(n)} \,\overset{\eqref{bound1}}{\longrightarrow}\, 0 \,\, . \label{k>r(1)} \eeq Moreover, by changing indexes in \eqref{st1inf}, we get \beq \lim_{n\rar+\infty}[i_1i_2k]^{(n)}p_{i_1k}^{(n)}p_{i_1i_2}^{(n)}=0 \quad \text{ for any } k \in I^{(\infty)}_{\geq2} \,\, . \label{st2inf} \eeq So, from \eqref{limsecij}, \eqref{k>r(1)} and \eqref{st2inf}, we get for any fixed $i_1,i_2 \in I^{(\infty)}_1$, $i_1<i_2$ \beq \lim_{n\rar+\infty}\Bigg\{\sum_{k \in I^{(\infty)}_1}\Big([i_1i_2k]^{(n)}p_{i_1k}^{(n)}p_{i_1i_2}^{(n)}+a_{i_1i_2k}^{(n)}\Big)\Bigg\}=0 \,\, . \label{hatPi1i2} \tag{$\triangle_{i_1i_2}$}\eeq Let us notice that \begin{multline}
\sum_{k \in I^{(\infty)}_1}\Big([i_1i_2k]^{(n)}p_{ik}^{(n)}p_{i_1i_2}^{(n)}+a_{i_1i_2k}^{(n)}\Big)=  \\
=\sum_{k=1}^{i_1}[i_1i_2k]^{(n)}\Bigg(\frac{(p_{i_2i_1}^{(n)}-1)^2(p_{i_1k}^{(n)})^2+2(p_{i_2i_1}^{(n)}+1)p_{i_1k}^{(n)}-3}{p_{i_2i_1}^{(n)}p_{i_1k}^{(n)}}\Bigg)+\sum_{k=i_1+1}^{r(1)}[i_1i_2k]^{(n)}p_{i_1k}^{(n)}p_{i_1i_2}^{(n)}+\sum_{k=i_1+1}^{r(1)}a_{i_1i_2k}^{(n)} \,\, . \label{scomp}
\end{multline} Furthermore, if $k\leq i_1<i_2$, then $p_{i_2i_1}^{(n)}, p_{i_1k}^{(n)} \geq 1$ by \eqref{goodorder} and hence \beq \frac{(p_{i_2i_1}^{(n)}-1)^2(p_{i_1k}^{(n)})^2+2(p_{i_2i_1}^{(n)}+1)p_{i_1k}^{(n)}-3}{p_{i_2i_1}^{(n)}p_{i_1k}^{(n)}}\geq1 \quad \text{ for any } k\leq i_1<i_2 \,\, .  \label{>1} \eeq

For any $i_1 \in \{1,{\dots},r(1){-}1\}$ and for any $q \in \{0,{\dots},r(1){-}i_1{-}1\}$, we set $i_2=r(1){-}q$ and we consider the following claim, which we denote by $\hat{P}(i_1,q)$: the limit $a_{i_1(r(1){-}q)k}^{(\infty)}$ exists for any $k \in \{i_1{+}1,{\dots},r(1)\}$ and $a_{i(r(1){-}q)k}^{(\infty)}=0$. 

First, we are going to prove that $\hat{P}(i_1,0)$ holds for any $1\leq i_1 \leq r(1){-}1$. By the very definition \eqref{aijk}, it follows that each $a_{i_1r(1)k}^{(n)}$, with $i_1{+}1 \leq k \leq r(1)$ is definitely non negative. Hence, by applying $(\triangle_{i_1r(1)})$ and \eqref{scomp}, we get the claim.

Let us fix now $1 \leq i_1 \leq r(1){-}1$ and $0\leq q \leq r(1){-}i{-}2$ and assume that $\hat{P}(i_1,q')$ holds for any $0 \leq q' \leq q$. By means of $(\triangle_{i_1(r(1){-}q')})$ and \eqref{scomp}, we get $a_{i_1(r(1){-}q')k}^{(\infty)}=0$ for any $i_1{+}1\leq k\leq r(1)$. Again, for any $i_1{+}1\leq k\leq r(1)$, we have: \begin{itemize}
\item[$\bcdot$] if $p_{(r(1){-}q{-}1)k}^{(\infty)} \in [1,+\infty]$, then, by the very definition \eqref{aijk}, we directly get that $a_{i_1(r(1){-}q{-}1)k}^{(n)}$ is definitely non negative;
\item[$\bcdot$] if $p_{(r(1){-}q{-}1)k}^{(\infty)} \in [0,1)$, then, by \eqref{goodorder}, it follows that there exists $1\leq q' \leq q$ such that $k=r(1){-}q'$ and so the limit $a_{i_1(r(1){-}q{-}1)k}^{(\infty)}$ exists and $a_{i_1(r(1){-}q{-}1)k}^{(\infty)}=0$.
\end{itemize} By means of $(\triangle_{i(\ell{-}q{-}1)})$, this actually proves that $\hat{P}(i_1,q{+}1)$ holds. Hence, we proved by induction that $\hat{P}(i,q)$ holds for any $1\leq i_1 \leq r(1)-1$, $0\leq q \leq r(1){-}i_1{-}1$. In particular by \eqref{scomp} we obtain \begin{align*} 
\eqref{hatPi1i2} \quad &\Longleftrightarrow \quad \left\{\begin{array}{ll}
\displaystyle \lim_{n\rar+\infty}[i_ii_2k]^{(n)}\Bigg(\frac{(p_{i_2i_1}^{(n)}-1)^2(p_{i_1k}^{(n)})^2+2(p_{i_2i_1}^{(n)}+1)p_{i_1k}^{(n)}-3}{p_{i_2i_1}^{(n)}p_{i_1k}^{(n)}}\Bigg)=0 & \, , \,\, 1\leq k <i_1 \\
\displaystyle \lim_{n\rar+\infty}[i_1i_1i_2]^{(n)}p_{i_2i_1}^{(n)}=0 \\
\displaystyle \lim_{n\rar+\infty}[i_1i_2k]^{(n)}p_{i_1k}^{(n)}p_{i_1i_2}^{(n)}=0 & \, , \,\, i_1{+}1\leq k \leq r(1) \\
a_{i_1i_2k}^{(\infty)}=0 & \, , \,\, i_1{+}1\leq k \leq r(1) 
\end{array}\right. \\
\quad\quad\quad\quad\,\, &\Longrightarrow \;\quad \left\{\begin{array}{ll}
\displaystyle \lim_{n\rar+\infty}[i_1i_2k]^{(n)}=0 & \, , \,\, 1 \leq k \leq i_1 \\
\displaystyle \lim_{n\rar+\infty}[i_1i_1i_2]^{(n)}p_{i_2i_1}^{(n)}=0 \\
\displaystyle \lim_{n\rar+\infty}[i_1i_2k]^{(n)}\big(p_{i_2k}^{(n)}-1\big)=0 & \, , \,\, i_1{+}1\leq k \leq r(1) \\
\end{array}\right. \,\, . \end{align*} Therefore, we get \begin{align}
&\lim_{n\rar+\infty}[i_1i_2i_3]^{(n)}p_{i_3i_2}^{(n)}=0 \quad \text{ for any } i_1,i_2,i_3 \in I^{(\infty)}_1 \, , \,\, i_1 \leq i_2 < i_3 \,\, . \label{bound4}
\end{align}

{\it Step 3.} We are going to apply \eqref{limseci}. Notice that, by changing indexes in \eqref{bound1}, it holds \beq \lim_{n\rar+\infty}[iik]^{(n)}p_{ki}^{(n)}=0 \quad \text{ for any $i \in I^{(\infty)}_1$, $k\in I^{(\infty)}_{\geq2}$ } \,\, . \label{st3inf} \eeq Therefore from \eqref{limseci} and \eqref{st3inf} we directly get \beq \lim_{n\rar+\infty}\Bigg\{\sum_{k \in I^{(\infty)}_1\setminus\{i\}}[iik]^{(n)}\bigg(p_{ki}^{(n)}-\frac43\bigg)\Bigg\} = \frac43d_ic_i^{(\infty)}+\frac13[iii]^{(\infty)} \,\, , \quad i \in I^{(\infty)}_1 \,\, . \label{hatPi} \tag{$\triangledown_i$} \eeq By applying \eqref{bound4} it follows that for any $i \in I_1^{(\infty)}$ all the summands inside the curly brackets in the left-hand side of \eqref{hatPi} are infinitesimal or definitely non positive, while all the summands in the right-hand side are non negative. Hence, it holds necessarily \beq c_{i_1}^{(\infty)}=0 \,\, , \quad [i_1i_1i_2]^{(\infty)}=0 \quad \text{ for any $i_1, i_2 \in I^{(\infty)}_1$ } \,\, . \label{bound5} \eeq

The thesis follows now from \eqref{bound1}, \eqref{bound2}, \eqref{bound4} and \eqref{bound5}. \end{proof}

Next, we aim to extend Theorem \ref{preltheo} by considering not only the most shrinking direction, but all the shrinking directions of $(g^{(n)})$. First, we need the following 

\begin{prop}[\cite{B\"o1}, Lemma 5.55] Assume that there exists a flag of $\fH$-subalgebras $\z=(\gk_1,{\dots},\gk_p)$ such that $v^{(\infty)} \in \sW^\S(\z)$. If $\gk_q$ is toral for some $1\leq q \leq p$, then \beq \scal\big(g^{(n)}\big) \leq \frac12\sum_{i>r(q)}d_ib_i^{(n)}e^{-t^{(n)}v_i^{(n)}}-\frac14\sum_{i,j,k >r(q)}[ijk]^{(n)}e^{t^{(n)}(v_i^{(n)}-v_j^{(n)}-v_k^{(n)})} \,\, , \label{scaltorest} \eeq where the application $r: s \mapsto r(s)$ is defined in \eqref{r}. \label{propscaltorest} \end{prop}

Since the estimate \eqref{scaltorest} plays a fundamental role in the proof of our main results, we present a proof of Proposition \ref{propscaltorest} in Appendix \ref{appendixA}. \smallskip

Let us consider $p \in \{1,{\dots},\ell_{v^{(\infty)}}{-}1\}$ in such a way that $\l_{r(p-1)+1}^{(n)}$ is bounded and $\l_{r(p)+1}^{(n)} \rar +\infty$. We set $I^{\rm gb}\=\cup_{q=1}^pI^{(\infty)}_q=\{1,{\dots},r(p)\}$ and we call it {\it index set of the generalized bounded eigenvalues of $(g^{(n)})$}. This name is due to the fact that for any $i \in I$, if  $\l_i^{(n)}$ is bounded then $i \in I^{\rm gb}$. Notice that it can happen that $\l_i^{(n)}\to+\infty$ for some $i \in I^{\rm gb}$.

Let also $I^{\rm sh}\=\{1,{\dots},\wt{r}\} \subsetneq I$ be the {\it index set of the shrinking eigenvalues of $(g^{(n)})$}, i.e. $\l_{\wt{r}}^{(n)} \rar 0$ and $\l_{\wt{r}+1}^{(n)}$ is bounded away from zero. We define then \beq \begin{gathered} \gk_1\=\gh+\gm^{(\infty)}_{I^{(\infty)}_1} \,\, , \quad \gk_2\=\gk_1+\gm^{(\infty)}_{I^{(\infty)}_2} \,\, , \quad {\dots} \quad , \quad \gk_{p-1}\=\gk_{p-2}+\gm^{(\infty)}_{I^{(\infty)}_{p-1}} \,\, , \\ \gl'\=\gk_p\= \gk_{p-1}+\gm^{(\infty)}_{I^{(\infty)}_p}=\gh + \sum_{i \in I^{\rm gb}}\gm^{(\infty)}_i \label{flaginfty} \end{gathered} \eeq and also \beq \gl \= \gh + \sum_{i \in I^{\rm sh}}\gm^{(\infty)}_i \,\, . \label{kinfty} \eeq  Notice that it necessary holds that $r(p-1) \leq \wt{r} \leq r(p)$, and hence $\gk_{p-1} \subset \gl \subset \gl'$. \smallskip

We are ready to prove our main result. Notice that both Theorem \ref{MAIN} and Theorem \ref{MAIN2} are consequences of the following

\begin{theorem} The set $\z\=(\gk_1,{\dots},\gk_{p-1},\gl')$ defined in \eqref{flaginfty} is a flag of $\fH$-subalgebras of $\gg$ and $v^{(\infty)} \in \eS^{\S}(\z)$. Moreover, the subspace $\gl$ defined in \eqref{kinfty} is a toral $\fH$-subalgebra of $\gg$ and the following conditions hold. \begin{itemize}
\item[\rm{A)}] For any $i\leq j \leq k$ such that $i\in I^{\rm sh}$, we have $$[ijk]^{(\infty)}=0 \quad \Longrightarrow \quad \lim_{n\rar+\infty}[ijk]^{(n)}p_{kj}^{(n)}=0 \,\, .$$
\item[\rm{B)}] For any $j,k \in I$ we have $$[I^{\rm sh}jk]^{(\infty)}>0 \quad \Longrightarrow \quad \lim_{n\rar+\infty}p_{kj}^{(n)}=1 \,\, .$$
\end{itemize} Finally, if $\gl'$ is toral, e.g. if $\gl=\gl'$, then $\displaystyle \lim_{n\rar+\infty}\scal\!\big(g^{(n)}\big)\leq0$. \label{main2} \end{theorem}

\begin{proof} If $p=1$, i.e. if $\l_{r(1)+1}^{(n)}\rar +\infty$, then the first part of the theorem coincide with the statement of Theorem \ref{preltheo}. Let us suppose then that $p>1$. If $p=2$, one can skip the next part of the proof.

We suppose now that $p>2$. For any $q \in \{1,{\dots},p-1\}$, we consider the following claim, which we denote by $\tilde{P}(q)$: $\gk_q$ is a toral $\fH$-subalgebra, $v^{(\infty)} \in \sW^{\S}(\gk_1,{\dots},\gk_q)$ and both (A), (B) hold after having replaced the index set $I^{\rm sh}$ with $I^{(\infty)}_q$.

Notice that $\tilde{P}(1)$ follows directly from Theorem \ref{preltheo}. Let us fix now $1\leq q \leq p-2$ and assume that $\tilde{P}(q')$ holds for any $1 \leq q' \leq q$. From \eqref{scaltorest}, it follows that \begin{align*}
\scal\!\big(g^{(n)}\big) &\leq \frac12\sum_{i>r(q)}d_ib_i^{(n)}e^{-t^{(n)}v_i^{(n)}}-\frac14\sum_{i,j,k >r(q)}[ijk]^{(n)}e^{t^{(n)}(v_i^{(n)}-v_j^{(n)}-v_k^{(n)})} \\
&\leq \frac14\bigg(2\sum_{i>r(q)}d_ib_i^{(n)}-\sum_{i,j,k>r(q)}[ijk]^{(n)}e^{t^{(n)}\big(v^{(n)}_i-v^{(n)}_j-v^{(n)}_k+v^{(n)}_{r(q)+1}\big)}\bigg)\frac1{\l_{r(q)+1}^{(n)}}
\end{align*} and so, since by assumption $\scal\!\big(g^{(n)}\big)$ is bounded from below, there exists necessarily $C>0$ such that $$\sum_{i,j,k>r(q)}[ijk]^{(n)}e^{t^{(n)}\big(v^{(n)}_i-v^{(n)}_j-v^{(n)}_k+v^{(n)}_{r(q)+1}\big)} \leq C \,\, \text{ for any $n \in \bN$ } \,\, . \label{boundscal2}$$ Then, by arguing as at the beginning of the proof of Theorem \ref{preltheo}, we directly get \beq i,j,k>r(q) \, , \,\, [ijk]^{(\infty)}>0 \quad \Longrightarrow \quad v_i^{(\infty)}-v_j^{(\infty)}-v_k^{(\infty)}+\hat{v}_{q+1}^{(\infty)}\leq0 \,\, .\label{k2} \eeq As a consequence $\gk_{q+1}$ is an $\fH$-subalgebra of $\gg$ and $v^{(\infty)}\in\sW^{\S}(\gk_1,{\dots},\gk_{q+1})$. Since $\l_{r(q+1)}^{(n)} \rar 0$ as $n \rar +\infty$, for any $i,j \in I$ such that $i \in I^{(\infty)}_{q+1}$, $i<j$ it follows that \begin{gather} \sec_i(g^{(n)})\cdot \l_i^{(n)} = d_ic_i^{(n)}+\frac14[iii]^{(n)}+\sum_{k\in I\setminus\{i\}}[iik]^{(n)}-\frac34\sum_{k\in I\setminus\{i\}}[iik]^{(n)}p_{ki}^{(n)} \,\longrightarrow\, 0 \,\, , \label{limseci2} \\ \sec_{ij}(g^{(n)})\cdot 4\l_i^{(n)} = \sum_{k\in I}\big([ijk]^{(n)}p_{ik}^{(n)}p_{ij}^{(n)}+a_{ijk}^{(n)}\big) \,\longrightarrow\, 0 \,\, , \label{limsecij2} \end{gather} where $\sec_i(g^{(n)})$ and $\sec_{ij}(g^{(n)})$ are defined in \eqref{seci} and \eqref{secij}, respectively, and the coefficients $a_{ijk}^{(n)}$ were introduced in \eqref{aijk}. So, one can apply mutatis mutandis Step 1, Step 2 and Step 3 already seen in the proof of Theorem \ref{preltheo} to conclude that $\tilde{P}(q+1)$ holds. Hence, it follows by induction that $\tilde{P}(q)$ holds for any $1 \leq q \leq p-1$.

From now on, it does not matter if $p=2$ or $p>2$. Since $\gk_{p-1}$ is toral and $\l_{r(p-1)+1}^{(n)}$ is bounded, from \eqref{scaltorest} it follows that $\gl'$ is an $\fH$-subalgebra of $\gg$ and $v^{(\infty)} \in \eS^{\S}(\gk_1,{\dots},\gl')$. Moreover, by repeating once again Step 1, Step 2 and Step 3 letting the index $i$ run from $1$ to $\wt{r}$, one can prove that $\gl$ is a toral subalgebra and that both conditions (A), (B) hold true.

Finally, for the proof of the last claim, we do not assume anymore that $p>1$, i.e. we allow $p$ to be $1$. Let us suppose by contradiction that $\gl'$ is toral and $\scal(g^{(n)})>\d$ definitely, for some $\d>0$. By \eqref{scaltorest} it holds that for any $n$ large enough $$\frac12\sum_{i>r(p)}d_ib_i^{(n)}e^{-t^{(n)}v_i^{(n)}}-\frac14\sum_{i,j,k >r(p)}[ijk]^{(n)}e^{t^{(n)}(v_i^{(n)}-v_j^{(n)}-v_k^{(n)})}>\d \,\, .$$ Hence, there exists a constant $C'>0$ such that \beq 4\d\l_{r(p)+1}^{(n)}+\sum_{i,j,k >r(p)}[ijk]^{(n)}e^{t^{(n)}(v_i^{(n)}-v_j^{(n)}-v_k^{(n)}+v_{r(p)+1}^{(n)})}<C' \quad \text{ for any $n \in \bN$ } \label{AAbsurd} \eeq which is clearly absurd, since all the terms in \eqref{AAbsurd} are non negative and $\l_{r(p)+1}^{(n)}$ is unbounded. \end{proof}

\subsection{An explicit example on $V_3(\bR^5)$} \label{V3} \hfill \par

We exhibit an example of a sequence of $\fSO(5)$-invariant metrics on the Stiefel manifold $V_3(\bR^5)$, i.e. the space of orthonormal $3$-frames in $\bR^5$, which diverges with bounded curvature. \smallskip

Let $M=V_3(\bR^5)=\fSO(5)/\fSO(2)$ and consider the inner product $Q(A_1,A_2)\={-}\frac12\Tr(A_1{\cdot}A_2)$ on $\so(5)$. We choose the $Q$-orthonormal basis for $\so(5)$ given by $$\begin{gathered} E\= e^4{\otimes}e_5{-}e^5{\otimes}e_4 \,\, , \quad X_1\= e^2{\otimes}e_3{-}e^3{\otimes}e_2 \,\, , \quad X_2\= e^3{\otimes}e_4{-}e^4{\otimes}e_3 \,\, , \quad X_3\= e^3{\otimes}e_5{-}e^5{\otimes}e_3 \,\, , \\
X_4\= e^2{\otimes}e_4{-}e^4{\otimes}e_2 \,\, , \quad X_5 \= e^2{\otimes}e_5{-}e^5{\otimes}e_2  \,\, , \quad X_6\= e^1{\otimes}e_4{-}e^4{\otimes}e_1 \,\, , \\
X_7 \= e^1{\otimes}e_5{-}e^5{\otimes}e_1 \,\, , \quad X_8 \= e^1{\otimes}e_3{-}e^3{\otimes}e_1 \,\, , \quad X_9 \= e^1{\otimes}e_2{-}e^2{\otimes}e_1 \,\, ,
\end{gathered}$$ where we denoted by $(e_1,{\dots},e_5)$ the standard basis of $\bR^5$ and by $(e^1,{\dots},e^5)$ its dual frame. Then, the isotropy algebra is $\so(2)=\vspan(E)$ and its $Q$-orthogonal reductive complement $\gm$ decomposes into six $\Ad(\fSO(2))$-irreducible submodules: $$\begin{gathered} \gm_1=\vspan(X_1) \,\, , \quad \gm_2=\vspan(X_2,X_3) \,\, , \quad \gm_3=\vspan(X_4,X_5) \,\, , \\ \gm_4=\vspan(X_6,X_7) \,\, , \quad \gm_5=\vspan(X_8) \,\, , \quad \gm_6=\vspan(X_9) \,\, . \end{gathered}$$ Notice that $\gm_2\simeq\gm_3\simeq\gm_4$ are equivalent to the standard representation of $\fSO(2)$, while $\gm_1\simeq\gm_5\simeq\gm_6$ are trivial. One can directly check that the coefficients related to this decomposition are \beq \begin{gathered} c_1=0 \,\, , \quad c_2=c_3=c_4=1 \,\, , \quad c_5=c_6=0 \,\, , \\ b_1=b_2=b_3=b_4=b_5=b_6=6 \,\, , \\ [123]=2 \,\, , \quad [156]=1 \,\, , \quad [245]=2 \,\, , \quad [346]=2 \,\, . \end{gathered}\eeq We define also $$\gk_1\=\gh+\gm_1 \simeq \so(2)\oplus\so(2) \,\, , \quad \gk_2\=\gk_1+\gm_2+\gm_3 \simeq \so(4) \,\, ,$$ which are $\fSO(2)$-subalgebras of $\so(5)$. We highlight that $\gk_1$ is toral, while $\gk_2$ is non-toral.

Let us consider the sequence $(g^{(n)})\subset\sM^{\SO(5)}_1$ defined by \beq g^{(n)}\= \textstyle\frac1{4n^4}Q_{\gm_1}+Q_{\gm_2}+Q_{\gm_3}+nQ_{\gm_4}+2nQ_{\gm_5}+2nQ_{\gm_6} \,\, . \label{seqex} \eeq Notice that the eigenvalues of the tangent direction $v^{(n)}$ are $$\begin{gathered} \textstyle
v_1^{(n)}=-\frac{2+4\log_2n}{\sqrt{20(\log_2n)^2+20\log_2n+6}} \,\, , \quad v_2^{(n)}=v_3^{(n)}=0 \,\, , \\ \textstyle
v_4^{(n)}=\frac{\log_2n}{\sqrt{20(\log_2n)^2+20\log_2n+6}} \,\, , \quad v_5^{(n)}=v_6^{(n)}=\frac{1+\log_2n}{\sqrt{20(\log_2n)^2+20\log_2n+6}}
\end{gathered}$$ and so $v^{(n)} \in \eS^{\S}(\gk_1)$, but $v^{(n)} \notin \eS^{\S}(\gk_1,\gk_2)$. From \eqref{expW} it follows that $(g^{(n)})$ lies in the space $\eM^{\fG}_1(\gk_1)$ of unit volume $\gk_1$-submersion metrics. One can directly check that the Ricci operator of $g^{(n)}$ is diagonal, with eigenvalues $$\begin{gathered}\textstyle \ric_1\!\big(g^{(n)}\big)=\frac{8n^2+1}{32n^6} \,\, , \quad \ric_2\!\big(g^{(n)}\big)=\ric_3\!\big(g^{(n)}\big)=\frac{14n^4+2n^2-1}{8n^4} \,\, , \\
\textstyle \ric_4\!\big(g^{(n)}\big)=-\frac{3n^2-6n+1}{2n^2} \,\, , \quad \ric_5\!\big(g^{(n)}\big)=\ric_6\!\big(g^{(n)}\big)=\frac{48n^6+48n^5-16n^4-1}{32n^6} \,\, .
\end{gathered}$$ By \cite[Thm 4]{BLS} it follows that $(g^{(n)})$ has bounded curvature. For the sake of thoroughness, we provide in Appendix \ref{appendixA} the explicit expression of all the components of the curvature operator $\Rm(g^{(n)})$.

This example shows that in some sense Theorem \ref{main2} is optimal. In fact, we have \beq p=2 \,\, , \quad I^{(\infty)}_1=I^{(\infty)}_{p-1}=I^{\rm sh}=\{1\} \,\, , \quad I^{(\infty)}_2=I^{(\infty)}_p=\{2,3\} \,\, , \quad I^{\rm gb}=\{1,2,3\} \,\, , \quad I^{(\infty)}_3=\{4,5,6\} \label{conclusion1} \eeq and so $\gl=\gk_1$, $\gl'=\gk_2$. Moreover \beq \textstyle [245]>0 \,\, ,\quad \frac{\l_5^{(n)}}{\l_4^{(n)}}=2\neq1 \,\, . \label{conclusion2}\eeq So, even though $v^{(\infty)} \in \eS^{\S}(\gl,\gl')$ because $$\textstyle v^{(\infty)}_1=-\frac{4}{\sqrt{20}} \,\, , \quad v^{(\infty)}_2=v^{(\infty)}_3=0 \,\, , \quad v^{(\infty)}_4=v^{(\infty)}_5=v^{(\infty)}_6=\frac1{\sqrt{20}} \,\, ,$$ from \eqref{conclusion2} it follows that claim (B) does not hold anymore if one replaces the index set $I^{\rm sh}$ with $I^{\rm gb}$. This means that $(g^{(n)})$ does not approach asymptotically a $\gl'$-submersion metric.

Moreover $$\textstyle \scal\!\big(g^{(n)}\big)=\frac{224n^6+288n^5-32n^4-8n^2-1}{32n^6} \rar 7>0$$ and this shows that {\it it is possible for a sequence of invariant metrics to diverge with bounded curvature and positive scalar curvature bounded away from zero.}

Finally, along the geodesic $\g_{v^{(n)}}(t)$ we have $$\begin{aligned} \scal(\g_{v^{(n)}}(t))=12-2e^{t(v^{(n)}_5-v^{(n)}_4)}-e^{tv^{(n)}_1}-6e^{-tv^{(n)}_4}-6e^{-tv^{(n)}_5}&-{\textstyle\frac12}e^{-t(2v^{(n)}_5-v^{(n)}_1)}- \\ &-2e^{-t(v^{(n)}_4+v^{(n)}_5)}-2e^{-t(v^{(n)}_5-v^{(n)}_4)} \end{aligned}$$ and so $\lim_{t\rar+\infty}\scal(\g_{v^{(n)}}(t))=-\infty$ for any $n \in \bN$. On the other hand, one can directly check that along the limit geodesic $\g_{v^{(\infty)}}(t)$, the Ricci operator is diagonal with eigenvalues $$\begin{gathered} \ric_1(\g_{v^{(\infty)}}(t))=e^{tv^{(\infty)}_1}{+}{\textstyle\frac12}e^{-t(2v^{(\infty)}_4{-}v^{(\infty)}_1)} \,\, , \quad \ric_2(\g_{v^{(\infty)}}(t))=\ric_3(\g_{v^{(\infty)}}(t))=2{-}{\textstyle\frac12}e^{tv^{(\infty)}_1}{+}{\textstyle\frac12}e^{-2tv^{(\infty)}_4} \\ \ric_4(\g_{v^{(\infty)}}(t))=3e^{-tv^{(\infty)}_4}{-}e^{-2tv^{(\infty)}_4} \,\, , \\ \ric_5(\g_{v^{(\infty)}}(t))=\ric_6(\g_{v^{(\infty)}}(t))=3e^{-tv^{(\infty)}_4}{-}e^{-2tv^{(\infty)}_4}{-}{\textstyle\frac12}e^{-t(2v^{(\infty)}_4-v^{(\infty)}_1)} \end{gathered}$$ and so, by applying again \cite[Thm 4]{BLS}, $\big|\Rm(\g_{v^{(\infty)}}(t))\big|_{\g_{v^{(\infty)}}(t)}$ is bounded. We highlight that the limit values of the Ricci eigenvalues along the original sequence $(g^{(n)})$ are $$\begin{gathered} \lim_{n\rar+\infty}\ric_1\!\big(g^{(n)}\big)=0 \,\, , \quad \lim_{n\rar+\infty}\ric_2\!\big(g^{(n)}\big)=\lim_{n\rar+\infty}\ric_3\!\big(g^{(n)}\big)={\textstyle\frac74} \,\, , \\
\lim_{n\rar+\infty}\ric_4\!\big(g^{(n)}\big)=-{\textstyle\frac32} \,\, , \quad \lim_{n\rar+\infty}\ric_5\!\big(g^{(n)}\big)=\lim_{n\rar+\infty}\ric_6\!\big(g^{(n)}\big)={\textstyle\frac32} \,\, ,
\end{gathered}$$ while along the limit geodesic $\g_{v^{(\infty)}}(t)$ $$\begin{gathered} \lim_{t\rar+\infty}\ric_1(\g_{v^{(\infty)}}(t))=0 \,\, , \quad \lim_{t\rar+\infty}\ric_2(\g_{v^{(\infty)}}(t))=\lim_{t\rar+\infty}\ric_3(\g_{v^{(\infty)}}(t))=2 \,\, , \\ \lim_{t\rar+\infty}\ric_4(\g_{v^{(\infty)}}(t))=\lim_{t\rar+\infty}\ric_5(\g_{v^{(\infty)}}(t))=\lim_{t\rar+\infty}\ric_6(\g_{v^{(\infty)}}(t))=0 \,\, . \end{gathered}$$ This actually shows that a diverging sequence $(g^{(n)}) \subset \eM^{\fG}_1$ with bounded curvature and limit direction $v^{(\infty)}$ can develop a different asymptotic behavior with respect to to the geodesic $\g_{v^{(\infty)}}(t)$. \smallskip

Finally, let us mention that in our previous example $\wt{r}=r(p-1)$. It is also easy to exhibit examples where $\wt{r}=r(p)$, e.g. by considering again Berger spheres as in Example \ref{Berger}. However, it is not clear whether it is actually possible to construct a sequence of invariant metrics which diverges with bounded curvature with $r(p-1)<\wt{r}<r(p)$. We highlight that for this to be the case it is necessary that the limit direction $v^{(\infty)}$ admits the eigenvalue $\hat{v}^{(\infty)}_p=0$ and the module $\gm_{I^{(\infty)}_p}$ needs to be $\Ad(\fK_{p-1})$-reducible.

\smallskip

\section{Algebraically collapsed sequences of $\fG$-invariant metrics} \label{section5} \setcounter{equation} 0

In this last section, we are going to apply Theorem \ref{MAIN} to give a characterization of algebraically collapsed sequences of invariant metrics on a given compact homogeneous manifold. In general, this is a major object of interest in the study of equivariant convergence of homogeneous Riemannian spaces. Although we do not investigate here such a topic, we refer to \cite{Heb,BWZ,Lau1,Lau2} for what concerns the theory of convergence of homogeneous Riemannian spaces and to \cite[Sec 9]{BL} for a detailed explication of the phenomenon of {\it algebraic collapse}. \smallskip

Let $M=\fG/\fH$ be a compact, connected and almost effective $m$-dimensional homogeneous space, with $\fG$ and $\fH$ compact Lie groups. We fix $Q$, and hence $\gm$, as in Section \ref{section2}. For the sake of notation, we set $$\mu \in \L^2\gg^* \otimes \gg \,\, , \quad \mu(X,Y) \= [X,Y]$$ and we decompose it by using the $Q$-orthogonal projection onto $\gh$ and $\gm$ as \beq \mu = (\mu|_{\gh \wedge \gg})+\mu_{\gh}+\mu_{\gm} \,\, , \quad \text{ with } \quad \mu_{\gh}: \gm \wedge \gm \rar \gh \,\, , \quad \mu_{\gm}: \gm \wedge \gm \rar \gm \,\, . \eeq

Let now $g \in \eM^{\fG}$ and $\f \in \eF^{\fG}$ be a good decomposition for $g$, i.e. it takes the form \eqref{diag}. We set $Q_{\gh} \= Q|_{\gh\otimes\gh}$. Let also $(e_{\a})$ be a $\f$-adapted $Q_{\gm}$-orthonormal basis for $\gm$ and $(z_{\g})$ be a $Q_{\gh}$-orthonormal basis for $\gh$. Then, the direct sum $Q_{\gh}+g$ is an $\Ad(\fH)$-invariant inner product on the whole Lie algebra $\gg$ with respect to which \beq|\mu|_{Q_{\gh}+g}^2 = \big|(\mu|_{\gh \wedge \gg})\big|_{Q_{\gh}+g}^2+|\mu_{\gh}|_{Q_{\gh}+g}^2+|\mu_{\gm}|_g^2 \,\, . \label{normbracket1} \eeq Notice that \beq \begin{aligned}
\big|(\mu|_{\gh \wedge \gg})\big|_{Q_{\gh}+g}^2 &=\big|(\mu|_{\gh \wedge \gh})\big|_{Q_{\gh}}^2+\sum_{i\in I}\sum_{\substack{e_{\a} \in \gm_i \\ z_{\g} \in \gh}}\Big|{\textstyle \Big[z_{\g},\frac{e_{\a}}{\sqrt{\l_i}}\Big]_{\gm_i}}\Big|_g^2 \\
&=\big|(\mu|_{\gh \wedge \gh})\big|_{Q_{\gh}}^2+\sum_{i\in I}\sum_{\substack{e_{\a} \in \gm_i \\ z_{\g} \in \gh}}\big|[z_{\g},e_{\a}]\big|_Q^2 \\
&=\big|(\mu|_{\gh \wedge \gh})\big|_{Q_{\gh}}^2+\sum_{i\in I}d_ic_i \,\, ,
\end{aligned} \eeq and so the norm $\big|(\mu|_{\gh \wedge \gg})\big|_{Q_{\gh}+g}$ does not depend on $g$. On the other hand \beq \begin{gathered}
|\mu_{\gh}|_{Q_{\gh}+g}^2=\sum_{i,j \in I}\sum_{\substack{e_{\a} \in \gm_i \\ e_{\b} \in \gm_j}}\Big|{\textstyle \Big[\frac{e_{\a}}{\sqrt{\l_i}},\frac{e_{\b}}{\sqrt{\l_j}}\Big]_{\gh}}\Big|_{Q}^2=\sum_{i\in I}\frac1{\l_i}\sum_{e_{\a},e_{\a'}\in \gm_i}\big|[e_{\a},e_{\a'}]_{\gh}\big|_Q^2=\sum_{i\in I}\frac{d_ic_i}{\l_i} \,\, ,\\
|\mu_{\gm}|_{g}^2=\sum_{i,j,k \in I}\sum_{\substack{e_{\a} \in \gm_i \\ e_{\b} \in \gm_j}}\Big|{\textstyle \Big[\frac{e_{\a}}{\sqrt{\l_i}},\frac{e_{\b}}{\sqrt{\l_j}}\Big]_{\gm_k}}\Big|_g^2=\sum_{i,j,k \in I}[ijk]_{\f}\frac{\l_k}{\l_i\l_j} \,\, .
\end{gathered} \label{normbracket2} \eeq

\begin{definition} A sequence $\big(g^{(n)}) \subset \eM^{\fG}$ of $\fG$-invariant metrics on $M$ is said to be {\it algebraically non-collapsed} if there exists $C>0$ such that $$|\mu_{\gh}|_{Q_{\gh}+g^{(n)}}^2+|\mu_{\gm}|_{g^{(n)}}^2<C \quad \text{ for any $n \in \bN$ } \, ,$$ otherwise it is said to be {\it algebraically collapsed}. \end{definition}

Notice that any sequence which lies in a compact subset of $\eM^{\fG}_1$ is never algebraically collapsed. By assuming that the fundamental group $\pi_1(M)$ is finite, the converse assertion also holds true. In fact, we prove now Proposition \ref{MAINcor2} by using Theorem \ref{main2}.

\begin{proof}[Proof of Proposition \ref{MAINcor2}] Since $M$ is connected and the fundamental group $\pi_1(M)$ is finite, up to enlarging the space $\eM^{\fG}$ of invariant metrics, we can assume that the group $\fG$ is connected and semisimple. Let us fix a sequence $(g^{(n)}) \subset \eM^{\fG}_1$ which diverges with bounded curvature. From now until the end of the proof, we adopt the notation introduced in Section \ref{section4}. By Lemma \ref{lemmablue} and Theorem \ref{main2}, we can choose $i_{\zero} \in I^{\rm sh}$ and $j_{\zero},s_{\zero} \in I\setminus I^{\rm sh}$ such that $[i_{\zero}j_{\zero}s_{\zero}]^{(\infty)}>0$. Then, by Theorem \ref{main2} and \eqref{normbracket2} we directly get $$|\mu_{\gm}|_{g^{(n)}}^2 \geq [i_{\zero}j_{\zero}s_{\zero}]^{(n)}\frac{\l_{s_{\zero}}^{(n)}}{\l_{i_{\zero}}^{(n)}\l_{j_{\zero}}^{(n)}} \sim [i_{\zero}j_{\zero}s_{\zero}]^{(\infty)}\frac1{\l_{i_{\zero}}^{(n)}} \rar +\infty$$ and so the claim follows. \end{proof}

The next easy example shows that the finiteness hypothesis on the fundamental group $\pi_1(M)$ cannot be removed.

\begin{example} Let $M^3=S^1{\times}S^2 = \fG/\fH$, with $\fG\=\fU(1){\times}\fSU(2)$ and $\fH\=\{1\}{\times}\fU(1) \subset \fG$. Let us fix an $\Ad(\fG)$-invariant inner product $Q$ on $\gg=\Lie(\fG)$ and a $Q$-orthonormal basis $(E,X_1,X_2,X_3)$ for $\gg$ such that $$\begin{gathered} \gg=\gh+\gm_1+\gm_2 \,\, , \quad \gh=\vspan(X_1) \, , \,\, \gm_1=\vspan(E) \, , \,\, \gm_2=\vspan(X_2,X_3) \,\, , \\
[E,X_i]=0 \,\, , \quad [X_1,X_2]=-2X_3 \,\, , \quad [X_2,X_3]=-2X_1 \,\, , \quad [X_3,X_1]=-2X_2 \,\, . 
\end{gathered}$$ We consider now the sequence of metrics $g^{(n)} \= \frac1{n^2}Q_{\gm_1}+nQ_{\gm_2}$, together with the $g^{(n)}$-normalized frame $$E^{(n)}\=nE \,\, , \quad X_2^{(n)}\={\textstyle\frac1{\sqrt{n}}}X_2 \,\, , \quad X_3^{(n)}\={\textstyle\frac1{\sqrt{n}}}X_3 \,\, .$$ Then, one can directly check that the curvature operator $\Rm(g^{(n)}): \L^2\gm \rar \L^2\gm$ is diagonal and explicitly given by $$ \Rm(g^{(n)})(E^{(n)}{\wedge}X^{(n)}_2)=\Rm(g^{(n)})(E^{(n)}{\wedge}X^{(n)}_3)=0 \,\, , \quad \Rm(g^{(n)})(X^{(n)}_2{\wedge}X^{(n)}_3)={\textstyle\frac4n}X^{(n)}_2{\wedge}X^{(n)}_3 \,\, ,$$ while $$[E^{(n)},X^{(n)}_2]=[E^{(n)},X^{(n)}_3]=0 \,\, , \quad [X^{(n)}_2,X^{(n)}_3]=-{\textstyle\frac2n}X_1 \,\, .$$ So, the sequence $(g^{(n)})$ diverges with bounded curvature and it is algebraically non-collapsed. \end{example}

Finally, let us consider a sequence $(g^{(n)}) \subset \eM^{\fG}$ and, up to a normalization, for any $n \in \bN$ fix the scale of the most shrinking direction to be $1$. This is equivalent of saying that, with respect to a diagonal decomposition as \eqref{diag(n)} in the previous section, $\min\!\big\{\l_1^{(n)},{\dots},\l_{\ell}^{(n)}\big\}=1$ for any $n \in \bN$. In this case, we say that {\it $(g^{(n)})$ is normalized with respect to the most shrinking direction}. Notice that any such a sequence is divergent if and only if $\vol(g^{(n)}) \rar +\infty$.

\begin{prop} If $(g^{(n)}) \subset \eM^{\fG}$ is normalized with respect to the most shrinking direction and has bounded curvature, then it is algebraically non-collapsed. \label{LASTprop}\end{prop}

\begin{proof} Let $(g^{(n)})$ be a divergent sequence of $\fG$-invariant metrics with bounded curvature and suppose that it is normalized with respect to the most shrinking direction. As in the proof of Proposition \ref{MAINcor2}, from now on we adopt the notation introduced at the beginning of Section \ref{section4}. By \eqref{ric}, the diagonal terms of the Ricci tensor along the sequence are given by \beq \ric_i(g^{(n)}) = \frac{b_i^{(n)}}{2\l_i^{(n)}}-\frac1{2d_i}\sum_{j,k \in I}[ijk]^{(n)}\frac{\l_k^{(n)}}{\l_i^{(n)}\l_j^{(n)}}+\frac1{4d_i}\sum_{j,k \in I}[ijk]^{(n)}\frac{\l_i^{(n)}}{\l_j^{(n)}\l_k^{(n)}} \,\, . \label{eigrick} \eeq Suppose by contradiction that $(g^{(n)})$ is algebraically collapsed. Since from our normalization $\l_i^{(n)}\geq 1$ for any $n\in \bN$, $1\leq i \leq \ell$, from \eqref{normbracket2} we get necessarily that $|\mu_{\gm}|_{g^{(n)}} \rar +\infty$. So, again by \eqref{normbracket2} there exists a triple $(i_1,i_2,i_3) \in I^3$ such that $[i_1i_2i_3]^{(n)}\frac{\l_{i_1}^{(n)}}{\l_{i_2}^{(n)}\l_{i_3}^{(n)}} \rar +\infty$. Since $\ric_{i_1}(g^{(n)})$ is bounded, by \eqref{eigrick} there exist $i_4, i_5 \in I$ such that $[i_1i_4i_5]^{(n)}\frac{\l_{i_4}^{(n)}}{\l_{i_1}^{(n)}\l_{i_5}^{(n)}} \rar +\infty$. By the way, $\ric_{i_4}(g^{(n)})$ is bounded too and then there exist $i_6, i_7 \in I$ such that $[i_4i_6i_7]^{(n)}\frac{\l_{i_6}^{(n)}}{\l_{i_4}^{(n)}\l_{i_7}^{(n)}} \rar +\infty$. Iterating this procedure, we obtain two sequences $(i_s), (j_s) \subset I$ such that $[i_sj_sj_{s+1}]^{(n)}\frac{\l_{j_{s+1}}^{(n)}}{\l_{i_s}^{(n)}\l_{j_s}^{(n)}} \rar +\infty$. Since $I=\{1,{\dots},\ell\}$ is finite and the relation defined on the set $\big\{\l_1^{(n)},{\dots},\l_{\ell}^{(n)}\big\}$ by $$a^{(n)} \prec b^{(n)} \,\iff\, \textstyle\frac{b^{(n)}}{a^{(n)}}\rar +\infty$$ is asymmetric and transitive, the sequences $(i_s)$ and $(j_s)$ are necessarily finite too, i.e. they are of the form $(i_1,{\dots},i_{s_{\zero}})$ and $(j_1,{\dots},j_{s_{\zero}})$, respectively. So, it follows that $\ric_{j_{s_{\zero}}}(g^{(n)}) \rar +\infty$ and this is absurd. \end{proof} 
\smallskip

\appendix

\section{} \label{appendixA} \setcounter{equation} 0

\subsection{Proof of Proposition \ref{propscaltorest}} \hfill \par

For convenience of the reader, we provide here a proof of Proposition \ref{propscaltorest} following B{\"o}hm's original approach. First, we need the following estimate.

\begin{prop} Let $\fG$ be a compact $N$-dimensional Lie group with a fixed $\Ad(\fG)$-invariant Euclidean inner product $Q$ on the Lie algebra $\gg \= \Lie(\fG)$, let $\ga \subset \gg$ be an abelian Lie subalgebra and let $\eB\=(e_1,{\dots},e_N)$ be a $Q$-orthonormal basis for $\gg$ such that $\ga=\vspan(e_1,{\dots},e_{q+1})$ for some $0\leq q \leq N-1$. Let also $\eB^{(n)}\=(e_1^{(n)},{\dots},e_N^{(n)})$ be a sequence of $Q$-orthonormal bases for $\gg$ such that $e_i^{(n)} \rar e_i$ as $n \rar +\infty$ for any $1 \leq i \leq N$. Then, there exist $\bar{n} \in \bN$ and $C>0$ such that \beq \sum_{i,j \leq q+1} Q\big([e_1^{(n)},e_i^{(n)}],e_j^{(n)}\big)^2 \leq C \sum_{\substack{i \leq q+1\\ k > q+1}}Q\big([e_1^{(n)},e_i^{(n)}],e_k^{(n)}\big)^2 \quad \text{ for any $n \geq \bar{n}$ } \, . \label{crucialest} \eeq \end{prop}
\begin{proof} Of course \eqref{crucialest} holds true if $\gg$ is abelian or $q=0,1$. Hence, we assume that $1<q<N-1$ and that $\gg$ is not abelian. Let $I\=\{1,{\dots},N\}$, $I_1\=\{2,{\dots},q+1\}$ and $I_2\= \{q+2,{\dots},N\}$. We highlights here that we will pass whenever convenient to a subsequence without mentioning it explicitly. Moreover, for any subspace $\gp \subset \gg$, we denote by $\gp^{\perp}$ its $Q$-orthogonal complement inside $\gg$.

Let us suppose by contradiction that \beq \sum_{i,j \in I_1} Q\big([e_1^{(n)},e_i^{(n)}],e_j^{(n)}\big)^2 > c^{(n)} \sum_{\substack{i \in I_1\\ k \in I_2}}Q\big([e_1^{(n)},e_i^{(n)}],e_k^{(n)}\big)^2 \quad \text{ for any $n \in \bN$ } \, , \label{absurd} \eeq for some sequence $c^{(n)}\rar +\infty$.

Let also $\gt \subset \gg$ be a maximal abelian Lie subalgebra of $\gg$ such that $e_1 \in \gt$. We claim that it is possible to assume that $e_1^{(n)} \in \gt$ for any $n \in \bN$. In fact, we can choose a sequence $(\gt^{(n)})$ of maximal abelian subalgebras of $\gg$ such that $e_1^{(n)} \in \gt^{(n)}$ and $\gt^{(n)} \rar \gt$ as $n \rar +\infty$. But then, there exists a sequence $(x^{(n)}) \subset \fG$ such that $\Ad(x^{(n)})(\gt^{(n)})=\gt$ and $x^{(n)} \rar 1_{\fG}$. Therefore, by setting $e_i'{}^{(n)} \=\Ad(x^{(n)})(e_i^{(n)})$ for any $i \in I$, we obtain a new $Q$-orthonormal basis $\eB'{}^{(n)}$ which converges to $\eB$. 

For any $i \in I_1$ we write \beq \gt^{\perp} \ni [e^{(n)}_1,e^{(n)}_i]=\sum_{j \in I_1\setminus\{i\}}a_{ij}^{(n)}e^{(n)}_j+z^{(n)}_i \,\, , \quad \text{ with $z^{(n)}_i \in \vspan(e^{(n)}_{q+2},{\dots},e^{(n)}_N) $} \label{aij} \eeq and we choose $j(i) \in I_1\setminus\{i\}$ such that $|a_{ij(i)}^{(n)}| \geq |a_{ij}^{(n)}|$ for any $j \in I_1\setminus\{i\}$, for any $n \in \bN$. Moreover, up to reorder the index set $I_1$, we may assume that $|a_{23}^{(n)}|\geq|a_{ij(i)}^{(n)}|$. So, by means of \eqref{absurd} and \eqref{aij}, we get \beq \big|a_{23}^{(n)}\big|^2 \geq \frac1q\sum_{i \in I_1}\big|a_{ij(i)}^{(n)}\big|^2 > \frac{c^{(n)}}{q^2} \sum_{i \in I_1}|z_i^{(n)}|_Q^2 \quad \text{ for any $n \in \bN$ } \, . \label{absurd2} \eeq
We claim now that it is possible to assume that for any $i \in I_1$ \beq \lim_{n \rar +\infty}\frac{\big|a_{ij(i)}^{(n)}\big|}{\big|a_{23}^{(n)}\big|}>0 \,\, . \label{limij(i)}\eeq In fact, let $I'_1 \= \{i \in I_1 : \text{ $i$ satisfies } \eqref{limij(i)}\}$ and $I''_1 \= I_1 \setminus I'_1$. Of course $\{2,3\} \subset I'_1$. Then, by \eqref{absurd2} \begin{align}
(1+|I''_1|)\big|a_{23}^{(n)}\big|^2 &= \big|a_{23}^{(n)}\big|^2 + \sum_{i \in I''_1}\frac{\big|a_{23}^{(n)}\big|^2}{\big|a_{ij(i)}^{(n)}\big|^2}\big|a_{ij(i)}^{(n)}\big|^2 \nonumber \\
&> \frac{c^{(n)}}{q^2} \sum_{i \in I_1}|z_i^{(n)}|_Q^2+\frac1q\sum_{\substack{i \in I''_1 \\ j \in I_1\setminus\{i\}}}\frac{\big|a_{23}^{(n)}\big|^2}{\big|a_{ij(i)}^{(n)}\big|^2}\big|a_{ij}^{(n)}\big|^2 \quad , \label{I'1(1)} \\
&\geq \tilde{c}^{(n)}\sum_{\substack{i \in I'_1 \\ k \in I''_1 \cup I_2}}Q\big([e_1^{(n)},e_i^{(n)}],e_k^{(n)}\big)^2 \nonumber
\end{align} where $$\tilde{c}^{(n)} \= \min\Bigg\{\frac{c^{(n)}}{q^2}, \frac1q\min_{i \in I''_1}\Bigg\{\frac{\big|a_{23}^{(n)}\big|^2}{\big|a_{ij(i)}^{(n)}\big|^2}\Bigg\}\Bigg\} \rar +\infty \,\, .$$ On the other hand \beq \sum_{i \in I'_1}\big|a_{ij(i)}^{(n)}\big|^2 \sim C'\, \big|a_{23}^{(n)}\big|^2 \label{I'1(2)} \quad \text{ for some $C'>0$ }\eeq and so by \eqref{I'1(1)} and \eqref{I'1(2)} we directly get that $$\sum_{i,j \in I'_1} Q\big([e_1^{(n)},e_i^{(n)}],e_j^{(n)}\big)^2 > \hat{c}^{(n)} \sum_{\substack{i \in I'_1\\ k \in I''_1\cup I_2}}Q\big([e_1^{(n)},e_i^{(n)}],e_k^{(n)}\big)^2 \quad \text{ for any $n \in \bN$ }$$ for some sequence $\hat{c}^{(n)} \rar +\infty$.

So, from now on, we assume $I_1=I'_1$ and hence $\big|a_{ij(i)}^{(n)}\big|>0$ for any $n \in \bN$, $i \in I_1$. Let also $d\=\dim(\gt)$ be the rank of $\gg$.

We are going to prove by induction that there exists a $Q$-orthonormal basis $(e_{1,1},e_{1,2},{\dots},e_{1,d})$ for $\gt$ and a set of vectors $E^{(\infty)}_i \in \ga\setminus\{0\}$, $i \in I_1$, such that for any $s \in \{1,{\dots},d\}$ the following claim, which we denote by $\bar{P}(s)$, holds: there exist a sequence $(e^{(n)}_{1,s}) \subset \vspan(e_{1,s},{\dots},e_{1,d}) \subset \gt$, with $e^{(n)}_{1,s} \rar e_{1,s}$ and, for any $i \in I_1$, a sequence of real numbers $\hat{a}^{(n)}_{i,s} >0$, with $\hat{a}^{(n)}_{i,s}\rar0$, such that, if we set $$e^{(n)}_{i,s}\= \begin{cases}
e^{(n)}_i & \text{ if $s=1$ } \\
{\rm pr}_{\gc_{\gg}(e_{1,1})\cap{\dots}\cap \gc_{\gg}(e_{1,s-1})}(e^{(n)}_i) & \text{ if $s>1$ }
\end{cases} \,\, ,$$ then \beq \frac1{\hat{a}^{(n)}_{i,s}}[e^{(n)}_{1,s},e^{(n)}_{i,s}] \,\rar\, E^{(\infty)}_i \,\, , \quad e^{(n)}_{i,s} \rar e_i \quad \text{ as $n \rar +\infty$ \, , \,\, for any $i \in I_1$ } \,\, . \label{inductivehyp}\eeq

First, we consider the case $s=1$ and we set $$e_{1,1} \= e_1 \,\, , \quad e^{(n)}_{1,1} \= e^{(n)}_{1} \,\, , \quad \hat{a}_{i,1}^{(n)} \= a_{ij(i)}^{(n)} \quad \text{ for any $i \in I_1$ } \,\, .$$ Next, we define $$E^{(n)}_{i,1} \= \frac1{\hat{a}^{(n)}_{i,1}}\sum_{j \in I_1\setminus\{i\}}a_{ij}^{(n)}\,e^{(n)}_{j,1} \,\, , \quad Z^{(n)}_{i,1} \= \frac1{\hat{a}^{(n)}_{i,1}}\,z_i^{(n)}$$ in such a way that \beq \frac1{\hat{a}^{(n)}_{i,1}}[e^{(n)}_{1,1},e^{(n)}_{i,1}]= E^{(n)}_{i,1}+Z^{(n)}_{i,1} \quad \text{ for any $i \in I_1$ } \, . \label{EZ} \eeq
By \eqref{absurd2} and \eqref{limij(i)}, it follows that $$\sum_{i \in I_1}\big|Z^{(n)}_{i,1}\big|_Q^2 \leq \e^{(n)} \quad \text{ for some $\e^{(n)}\rar 0$ } \, ,$$ while, by construction, $E^{(\infty)}_i \= \lim_{n\rar+\infty}E^{(n)}_{i,1} \neq 0$ and $E^{(\infty)}_i \in \ga \cap \gt^{\perp}$. Hence, it follows that $\bar{P}(1)$ holds. Let us fix now $1\leq s \leq d-1$ and assume that $\bar{P}(s')$ holds true for any $1 \leq s' \leq s$. Notice that, by the inductive hypothesis, we get $[e_{1,s'},e_i]=0$ for any $1 \leq s' \leq s$, $i \in I_1$ and then $\ga \subset \gc_{\gg}(e_{1,1})\cap{\dots}\cap \gc_{\gg}(e_{1,s})$. Here, we denoted by $\gc_{\gg}(X)$ the centralizer of $X \in \gg$ in $\gg$.

We consider now the following $Q$-orthogonal decompositions: $$\begin{gathered}
e^{(n)}_{1,s} \= \a^{(n)}_se_{1,s}+\tilde{e}^{(n)}_{1,s+1} \,\, , \\
e^{(n)}_{i,s} \= T^{(n)}_i+V^{(n)}_{i,s+1}+W^{(n)}_{i,s+1} \,\, , \quad i \in I_1 \,\, ,\end{gathered}$$ with $\tilde{e}^{(n)}_{1,s+1}\in \gt$ and $T^{(n)}_i \in \gt$, $V^{(n)}_{i,s+1} \in \gc_{\gg}(e_{1,1}) \cap {\dots} \cap \gc_{\gg}(e_{1,s}) \cap \gt^{\perp}$, $W^{(n)}_{i,s+1} \in \big(\gc_{\gg}(e_{1,1}) \cap {\dots} \cap \gc_{\gg}(e_{1,s})\big)^{\perp}$. Then $$[e_{1,s}^{(n)},e_{i,s}^{(n)}]=[\tilde{e}^{(n)}_{1,s+1},V^{(n)}_{i,s+1}]+[e^{(n)}_{1,s},W^{(n)}_{i,s+1}] \,\, ,$$ with $[\tilde{e}^{(n)}_{1,s+1},V^{(n)}_{i,s+1}] \in \gc_{\gg}(e_{1,1}) \cap {\dots} \cap \gc_{\gg}(e_{1,s}) \cap \gt^{\perp}$ and $[e_1^{(n)},W^{(n)}_{i,s+1}] \in \big(\gc_{\gg}(e_{1,1}) \cap {\dots} \cap \gc_{\gg}(e_{1,s})\big)^{\perp}$. If we set $$\wt{E}^{(n)}_{i,s} \= \frac1{\hat{a}^{(n)}_{i,s}}[e^{(n)}_{1,s},e^{(n)}_{i,s}] \,\, ,$$ we get \beq [\tilde{e}^{(n)}_{1,s+1},V^{(n)}_{i,s+1}]=\hat{a}^{(n)}_{i,s} \,{\rm pr}_{\gc_{\gg}(e_{1,1})\cap{\dots}\gc_{\gg}(e_{1,s})}(\wt{E}^{(n)}_{i,s}) \eeq and hence, since ${\rm pr}_{\gc_{\gg}(e_{1,1})\cap{\dots}\gc_{\gg}(e_{1,s})}(\wt{E}^{(n)}_{i,s}) \rar E^{(\infty)}_i \neq 0$ as $n \rar +\infty$, we deduce that $\tilde{e}^{(n)}_{1,s+1} \neq 0$. Next, we set $$e^{(n)}_{1,s+1} \= \frac{\tilde{e}^{(n)}_{1,s+1}}{|\tilde{e}^{(n)}_{1,s+1}|_Q} \,\, , \quad e_{1,s+1} \= \lim_{n\rar+\infty}e^{(n)}_{1,s+1} \,\, , \quad \hat{a}^{(n)}_{i,s+1}\=\frac{\hat{a}^{(n)}_{i,s}}{|\tilde{e}^{(n)}_{1,s+1}|_Q} \,\, .$$ Since $e^{(n)}_{i,s+1}=T^{(n)}_i+V^{(n)}_{i,s+1}$, it follows that $$\frac1{\hat{a}^{(n)}_{i,s+1}}[e^{(n)}_{1,s+1},e^{(n)}_{i,s+1}]={\rm pr}_{\gc_{\gg}(e_{1,1})\cap{\dots}\gc_{\gg}(e_{1,s})}(\wt{E}^{(n)}_{i,s}) = E^{(n)}_{i,s+1} + Z^{(n)}_{i,s+1} \,\, ,$$ where $$\begin{gathered} E^{(n)}_{i,s+1}\={\rm pr}_{\vspan(e^{(n)}_{2,s+1},{\dots},e^{(n)}_{q+1,s+1})}\big({\rm pr}_{\gc_{\gg}(e_{1,1})\cap{\dots}\gc_{\gg}(e_{1,s})}(\wt{E}^{(n)}_{i,s})\big) \,\, , \\ Z^{(n)}_{i,s+1}\={\rm pr}_{(\vspan(e^{(n)}_{2,s+1},{\dots},e^{(n)}_{q+1,s+1}))^{\perp}}\big({\rm pr}_{\gc_{\gg}(e_{1,1})\cap{\dots}\gc_{\gg}(e_{1,s})}(\wt{E}^{(n)}_{i,s})\big) \,\, . \end{gathered}$$ Since by inductive hypothesis $\ga \subset \gc_{\gg}(e_{1,1})\cap{\dots}\cap \gc_{\gg}(e_{1,s})$, it follows that $e^{(n)}_{i,s+1} \rar e_i$ for any $i \in I_1$ and hence $$E^{(n)}_{i,s+1} \rar E^{(\infty)}_i \,\, , \quad Z^{(n)}_{i,s+1} \rar 0 \quad \text{ as $n\rar+\infty$ } \,\, .$$ Since $[e_{1,s+1},e_i]=\hat{a}^{(\infty)}_{i,s+1}E^{(\infty)}_i$, with $\hat{a}^{(\infty)}_{i,s+1}\=\lim_{n\rar+\infty}\hat{a}^{(n)}_{i,s+1}$, and $e_i,E^{(\infty)}_i \in \ga$, it follows that $\hat{a}^{(\infty)}_{i,s+1}=0$. This proves that $\bar{P}(s{+}1)$ holds and hence, by induction that $\bar{P}(s)$ holds for any $1\leq s \leq d$.

By \eqref{inductivehyp}, it follows that $$[e_{1,s},e_i]=0 \,\, , \quad E^{(\infty)}_i \in \ga \cap \gt^{\perp} \quad \text{ for any $i \in I_1$,\, $1 \leq s \leq d$ } \, , \quad $$ and hence $[\gt,\ga]=\{0\}$, $\ga\cap\gt^{\perp}\neq\{0\}$. Therefore, $\gt+\ga$ is an abelian Lie subalgebra of $\gg$ and $\gt \subsetneq \gt+\ga$, which is clearly absurd since $\gt$ is maximal by assumption. \end{proof}

\begin{proof}[Proof of Proposition \ref{propscaltorest}] From now until the end of the proof, we adopt the notation introduced at the beginning of Section \ref{section4}. Assume that $v^{(\infty)} \in \eS^{\S}(\gk_1,{\dots},\gk_p)$ and that $\gk_q$ is toral for some $1\leq q \leq p$. From \eqref{scal} it follows directly that \begin{align*}
\scal\big(g^{(n)}\big) \!&= \frac12\sum_{i \in I}d_ib_i^{(n)}e^{-t^{(n)}v_i^{(n)}}-\frac14\sum_{i,j,k \in I}[ijk]^{(n)}e^{t^{(n)}(v_i^{(n)}-v_j^{(n)}-v_k^{(n)})} \\
&= \frac12\sum_{i\leq r(q)} e^{-t^{(n)}v_i^{(n)}}\Bigg\{\sum_{j,k \leq r(q)} [ijk]^{(n)}\Big(1-\frac12e^{t^{(n)}(v_j^{(n)}-v_k^{(n)})}\Big)+\sum_{\substack{j\leq r(q) \\ k > r(q)}} [ijk]^{(n)}\Big(2-\frac12e^{t^{(n)}(v_k^{(n)}-v_j^{(n)})}\Big)- \\
&\phantom{aaaaaa}-\sum_{j,k>r(q)}[ijk]^{(n)}\Big(\frac12e^{t^{(n)}(v_j^{(n)}-v_k^{(n)})}+\frac12e^{t^{(n)}(v_k^{(n)}-v_j^{(n)})}-1\Big)-\sum_{\substack{j\leq r(q) \\ k > r(q)}} [ijk]^{(n)}e^{t^{(n)}(v_j^{(n)}-v_k^{(n)})}- \\
&\phantom{aaaaaa}-\frac12\sum_{j,k>r(q)}[ijk]^{(n)}e^{t^{(n)}(2v_i^{(n)}-v_j^{(n)}-v_k^{(n)})}\Bigg\}+\frac12\sum_{i>r(q)}d_ib_i^{(n)}e^{-t^{(n)}v_i^{(n)}}- \\
&\phantom{aaaaaa}-\frac14\sum_{i,j,k >r(q)}[ijk]^{(n)}e^{t^{(n)}(v_i^{(n)}-v_j^{(n)}-v_k^{(n)})} \,\, .
\end{align*} Since $\gk_q$ is toral, it splits as $\gk_q = \gh + \ga$, with $[\gh,\ga]=[\ga,\ga]=\{0\}$ and $\ga \neq \{0\}$. Hence, from \eqref{crucialest}, it follows that there exist $\bar{n} \in \bN$ and a constant $C>0$ such that \beq \sum_{j,k \leq r(q)} [ijk]^{(n)} \leq C \sum_{\substack{j\leq r(q)\\k>r(q)}}[ijk]^{(n)} \quad \text { for any $n \geq \bar{n}$ , \, $1 \leq i \leq r(q)$ } \, .\eeq We can also assume that there exists $\e>0$ such that $v_k^{(n)}-v_j^{(n)}>\e$ for any $j \leq r(q)$, $k>r(q)$ and $n \geq \bar{n}$. Then \begin{align*}
\sum_{j,k \leq r(q)} [ijk]^{(n)}\Big(1-\frac12e^{t^{(n)}(v_j^{(n)}-v_k^{(n)})}\Big)+\sum_{\substack{j\leq r(q) \\ k > r(q)}} [ijk]^{(n)}\Big(2-&\frac12e^{t^{(n)}(v_k^{(n)}-v_j^{(n)})}\Big) \leq \\
&\leq \sum_{j,s \leq r(q)} [ijk]^{(n)}+\sum_{\substack{j\leq r(q) \\ k > r(q)}} [ijk]^{(n)}\Big(2-\frac12e^{t^{(n)}\e}\Big) \\
&\leq -\frac12\sum_{\substack{j\leq r(q) \\ k > r(q)}} [ijk]^{(n)}\Big(e^{t^{(n)}\e}-\tilde{C}\Big)
\end{align*} with $\tilde{C}\=2C+4$. Since $\frac12e^{t^{(n)}(v_j^{(n)}-v_k^{(n)})}+\frac12e^{t^{(n)}(v_k^{(n)}-v_j^{(n)})}\geq 1$, the claim follows. \end{proof}

\subsection{An explicit example on $V_3(\bR^5)$, part II} \hfill \par
We compute here the expression of the full curvature operator along the sequence $(g^{(n)})$ of unit volume invariant metrics on the Stiefel manifold $V_3(\bR^5)=\fSO(5)/\fSO(2)$ that we studied in Section \ref{section4}. Let us consider the $g^{(n)}$-orthonormal frame $$\begin{gathered} X_1^{(n)}\=2n^2X_1 \,\, , \quad X_2^{(n)}\=X_2 \,\, , \quad X_3^{(n)}\=X_3 \,\, , \quad X_4^{(n)}\=X_4 \,\, , \quad X_5^{(n)}\=X_5 \,\, , \\ X_6^{(n)}\=\textstyle\frac1{\sqrt{n}}X_6 \,\, , \quad X_7^{(n)}\=\frac1{\sqrt{n}}X_7 \,\, , \quad X_8^{(n)}\=\frac1{\sqrt{2n}}X_8 \,\, , \quad X_9^{(n)}\=\frac1{\sqrt{2n}}X_9 \,\, . \end{gathered}$$ Then, the curvature operator $\Rm(g^{(n)}): \L^2\gm \rar \L^2\gm$ takes the following form. $$\begin{array}{l}
\Rm(g^{(n)})(X_1^{(n)}{\wedge}X_2^{(n)})=\frac1{16n^4}X_1^{(n)}{\wedge}X_2^{(n)}+\frac{3n-1}{16\sqrt2 n^4}X_6^{(n)}{\wedge}X_9^{(n)} \\
\Rm(g^{(n)})(X_1^{(n)}{\wedge}X_3^{(n)})=\frac1{16n^4}X_1^{(n)}{\wedge}X_3^{(n)}+\frac{3n-1}{16\sqrt2 n^4}X_7^{(n)}{\wedge}X_9^{(n)} \\
\Rm(g^{(n)})(X_1^{(n)}{\wedge}X_4^{(n)})=\frac1{16n^4}X_1^{(n)}{\wedge}X_4^{(n)}-\frac{3n-1}{16\sqrt2 n^4}X_6^{(n)}{\wedge}X_8^{(n)} \\
\Rm(g^{(n)})(X_1^{(n)}{\wedge}X_5^{(n)})=\frac1{16n^4}X_1^{(n)}{\wedge}X_5^{(n)}-\frac{3n-1}{16\sqrt2 n^4}X_7^{(n)}{\wedge}X_8^{(n)} \\
\Rm(g^{(n)})(X_1^{(n)}{\wedge}X_6^{(n)})=\frac{2n^2+n-1}{16\sqrt2 n^4}X_2^{(n)}{\wedge}X_9^{(n)}-\frac{2n^2+n-1}{16\sqrt2 n^4}X_4^{(n)}{\wedge}X_8^{(n)} \\
\Rm(g^{(n)})(X_1^{(n)}{\wedge}X_7^{(n)})=\frac{2n^2+n-1}{16\sqrt2 n^4}X_3^{(n)}{\wedge}X_9^{(n)}-\frac{2n^2+n-1}{16\sqrt2 n^4}X_5^{(n)}{\wedge}X_8^{(n)} \\
\Rm(g^{(n)})(X_1^{(n)}{\wedge}X_8^{(n)})=\frac1{64n^6}X_1^{(n)}{\wedge}X_8^{(n)}-\frac{n-1}{8\sqrt2 n^3}X_4^{(n)}{\wedge}X_6^{(n)}-\frac{n-1}{8\sqrt2 n^3}X_5^{(n)}{\wedge}X_7^{(n)} \\
\Rm(g^{(n)})(X_1^{(n)}{\wedge}X_9^{(n)})=\frac1{64n^6}X_1^{(n)}{\wedge}X_9^{(n)}+\frac{n-1}{8\sqrt2 n^3}X_2^{(n)}{\wedge}X_6^{(n)}+\frac{n-1}{8\sqrt2 n^3}X_3^{(n)}{\wedge}X_7^{(n)} \\
\Rm(g^{(n)})(X_2^{(n)}{\wedge}X_3^{(n)})=X_2^{(n)}{\wedge}X_3^{(n)}+\frac{16n^4-1}{16n^4}X_4^{(n)}{\wedge}X_5^{(n)}-\frac{n^2-6n+1}{8n^2}X_6^{(n)}{\wedge}X_7^{(n)} \\
\Rm(g^{(n)})(X_2^{(n)}{\wedge}X_4^{(n)})=\frac{16n^4-3}{16n^4}X_2^{(n)}{\wedge}X_4^{(n)}+\frac{8n^4-1}{8n^4}X_3^{(n)}{\wedge}X_5^{(n)}-\frac{2n^5-12n^4+2n^3+1}{16n^5}X_8^{(n)}{\wedge}X_9^{(n)} \\
\Rm(g^{(n)})(X_2^{(n)}{\wedge}X_5^{(n)})=-\frac1{16n^4}X_3^{(n)}{\wedge}X_4^{(n)} \\
\Rm(g^{(n)})(X_2^{(n)}{\wedge}X_6^{(n)})=\frac{n-1}{8\sqrt2 n^3}X_1^{(n)}{\wedge}X_9^{(n)}-\frac{7n^2-2n-1}{8n^2}X_2^{(n)}{\wedge}X_6^{(n)}-\frac{n-1}{2n}X_3^{(n)}{\wedge}X_7^{(n)} \\
\Rm(g^{(n)})(X_2^{(n)}{\wedge}X_7^{(n)})=-\frac{(n+1)(3n-1)}{8n^2}X_3^{(n)}{\wedge}X_6^{(n)} \\
\Rm(g^{(n)})(X_2^{(n)}{\wedge}X_8^{(n)})=\frac{5n^2-2n+1}{8n^2}X_2^{(n)}{\wedge}X_8^{(n)}+\frac{8n^5+8n^4-1}{32n^5}X_4^{(n)}{\wedge}X_9^{(n)} \\
\Rm(g^{(n)})(X_2^{(n)}{\wedge}X_9^{(n)})=\frac{(n+1)(2n-1)}{16\sqrt2 n^4}X_1^{(n)}{\wedge}X_6^{(n)}+\frac{12n^5-16n^4+43+1}{32n^5}X_4^{(n)}{\wedge}X_8^{(n)} \\
\Rm(g^{(n)})(X_3^{(n)}{\wedge}X_4^{(n)})=-\frac1{16n^4}X_2^{(n)}{\wedge}X_5^{(n)} \\
\Rm(g^{(n)})(X_3^{(n)}{\wedge}X_5^{(n)})=-\frac{8n^4-1}{8n^4}X_2^{(n)}{\wedge}X_4^{(n)}+\frac{16n^4-3}{16n^4}X_3^{(n)}{\wedge}X_5^{(n)}-\frac{2n^5+2n^3-12n^4+1}{16n^5}X_8^{(n)}{\wedge}X_9^{(n)} \\
\Rm(g^{(n)})(X_3^{(n)}{\wedge}X_6^{(n)})=-\frac{(n+1)(3n-1)}{8n^2}X_2^{(n)}{\wedge}X_7^{(n)} \\
\Rm(g^{(n)})(X_3^{(n)}{\wedge}X_7^{(n)})=\frac{n+1}{8\sqrt2n^3}X_1^{(n)}{\wedge}X_9^{(n)}-\frac{n-1}{2n}X_2^{(n)}{\wedge}X_6^{(n)}-\frac{7n^2-2n-1}{8n^2}X_3^{(n)}{\wedge}X_7^{(n)} \\
\Rm(g^{(n)})(X_3^{(n)}{\wedge}X_8^{(n)})=\frac{5n^2-2n+1}{8n^2}X_3^{(n)}{\wedge}X_8^{(n)}+\frac{8n^5+8n^4-1}{32n^5}X_5^{(n)}{\wedge}X_9^{(n)} \\
\Rm(g^{(n)})(X_3^{(n)}{\wedge}X_9^{(n)})=\frac{(n+1)(2n-1)}{16\sqrt2n^4}X_1^{(n)}{\wedge}X_7^{(n)}+\frac{12n^5-16n^4+4n^3+1}{32n^5}X_5^{(n)}{\wedge}X_8^{(n)} \\
\Rm(g^{(n)})(X_4^{(n)}{\wedge}X_5^{(n)})=\frac{16n^4-1}{16n^4}X_2^{(n)}{\wedge}X_3^{(n)}+X_4^{(n)}{\wedge}X_5^{(n)}-\frac{n^2-6n+1}{8n^2}X_6^{(n)}{\wedge}X_7^{(n)} \\
\Rm(g^{(n)})(X_4^{(n)}{\wedge}X_6^{(n)})=-\frac{n-1}{8\sqrt2n^3}X_1^{(n)}{\wedge}X_8^{(n)}-\frac{7n^2-2n-1}{8n^2}X_4^{(n)}{\wedge}X_6^{(n)}-\frac{n-1}{2n}X_5^{(n)}{\wedge}X_7^{(n)} \\
\Rm(g^{(n)})(X_4^{(n)}{\wedge}X_7^{(n)})=-\frac{(n+1)(3n-1)}{8n^2}X_5^{(n)}{\wedge}X_6^{(n)} \\
\Rm(g^{(n)})(X_4^{(n)}{\wedge}X_8^{(n)})=-\frac{(n+1)(2n-1)}{16\sqrt2n^4}X_1^{(n)}{\wedge}X_6^{(n)}+\frac{12n^5-16n^4+4n^3+1}{32n^5}X_2^{(n)}{\wedge}X_9^{(n)} \\
\Rm(g^{(n)})(X_4^{(n)}{\wedge}X_9^{(n)})=-\frac{8n^5+8n^4-1}{32n^5}X_2^{(n)}{\wedge}X_8^{(n)}+\frac{5n^2-2n+1}{8n^2}X_4^{(n)}{\wedge}X_9^{(n)} \\
\Rm(g^{(n)})(X_5^{(n)}{\wedge}X_6^{(n)})=-\frac{(n+1)(3n-1)}{8n^2}X_4^{(n)}{\wedge}X_7^{(n)} \\
\Rm(g^{(n)})(X_5^{(n)}{\wedge}X_7^{(n)})=-\frac{n-1}{8\sqrt2n^3}X_1^{(n)}{\wedge}X_8^{(n)}-\frac{n-1}{2n}X_4^{(n)}{\wedge}X_6^{(n)}-\frac{7n^2-2n-1}{8n^2}X_5^{(n)}{\wedge}X_7^{(n)} \\
\Rm(g^{(n)})(X_5^{(n)}{\wedge}X_8^{(n)})=-\frac{(n+1)(2n-1)}{16\sqrt2n^4}X_1^{(n)}{\wedge}X_7^{(n)}+\frac{12n^5-16n^4+4n^3+1}{32n^5}X_3^{(n)}{\wedge}X_9^{(n)} \\
\Rm(g^{(n)})(X_5^{(n)}{\wedge}X_9^{(n)})=\frac{8n^5+8n^4-1}{32n^5}X_3^{(n)}{\wedge}X_8^{(n)}+\frac{5n^2-2n+1}{8n^2}X_5^{(n)}{\wedge}X_9^{(n)} \\
\Rm(g^{(n)})(X_6^{(n)}{\wedge}X_7^{(n)})=-\frac{n^2-6n+1}{8n^2}X_2^{(n)}{\wedge}X_3^{(n)}-\frac{n^2-6n+1}{8n^2}X_4^{(n)}{\wedge}X_5^{(n)}-\frac1{n}X_6^{(n)}{\wedge}X_7^{(n)} \\
\Rm(g^{(n)})(X_6^{(n)}{\wedge}X_8^{(n)})=-\frac{3n-1}{16\sqrt2n^4}X_1^{(n)}{\wedge}X_4^{(n)}+\frac{n^2+6n-3}{8n^2}X_6^{(n)}{\wedge}X_8^{(n)} \\
\Rm(g^{(n)})(X_6^{(n)}{\wedge}X_9^{(n)})=\frac{3n-1}{16\sqrt2n^4}X_1^{(n)}{\wedge}X_2^{(n)}+\frac{n^2+6n-3}{8n^2}X_6^{(n)}{\wedge}X_9^{(n)} \\
\Rm(g^{(n)})(X_7^{(n)}{\wedge}X_8^{(n)})=-\frac{3n-1}{16\sqrt2n^4}X_1^{(n)}{\wedge}X_5^{(n)}+\frac{n^2+6n-3}{8n^2}X_7^{(n)}{\wedge}X_8^{(n)} \\
\Rm(g^{(n)})(X_7^{(n)}{\wedge}X_9^{(n)})=\frac{3n-1}{16\sqrt2n^4}X_1^{(n)}{\wedge}X_3^{(n)}+\frac{n^2+6n-3}{8n^2}X_7^{(n)}{\wedge}X_9^{(n)} \\
\Rm(g^{(n)})(X_8^{(n)}{\wedge}X_9^{(n)})=-\frac{2n^5-12n^4+2n^3+1}{16n^5}X_2^{(n)}{\wedge}X_4^{(n)}-\frac{2n^5-12n^4+2n^3+1}{16n^5}X_3^{(n)}{\wedge}X_5^{(n)}+\frac{32n^5-3}{64n^6}X_8^{(n)}{\wedge}X_9^{(n)}
\end{array}$$

\bigskip\bigskip
\font\smallsmc = cmcsc8
\font\smalltt = cmtt8
\font\smallit = cmti8
\hbox{\parindent=0pt\parskip=0pt
\vbox{\baselineskip 9.5 pt \hsize=5truein
\obeylines
{\smallsmc
Dipartimento di Matematica e Informatica ``Ulisse Dini'', Universit$\scalefont{0.55}{\text{\Aac}}$ di Firenze
Viale Morgagni 67/A, 50134 Firenze, ITALY}
\smallskip
{\smallit E-mail adress}\/: {\smalltt francesco.pediconi@unifi.it
}
}
}


\begin{thebibliography}{20}

\bibitem{Bes} \textsc{A. L. Besse}, Einstein Manifolds, {\it Springer-Verlag}, {\it Berlin}, 2008.

\bibitem{B\"o1} \textsc{C. B\"ohm}, {\it Homogeneous Einstein metrics and simplicial complexes}, J. Differential Geom., {\bf 67} (2004), 79--165.

\bibitem{B\"o2} \textsc{C. B\"ohm}, {\it Non-existence of homogeneous Einstein metrics}, Comment. Math. Helv. {\bf 80} (2005), 123--146.

\bibitem{B\"o3} \textsc{C. B\"ohm}, {\it On the long time behavior of homogeneous Ricci flows}, Comment. Math. Helv. {\bf 90} (2015), 543--571.

\bibitem{BL} \textsc{C. Böhm, R. Lafuente}, {\it Immortal homogeneous Ricci flows}, Invent. math. \textbf{212} (2018), 461--529.

\bibitem{BLS} \textsc{C. Böhm, R. Lafuente, M. Simon}, {\it Optimal curvature estimates for homogeneous Ricci flows}, Int. Math. Res. Not. IMRN 2019 no. 13, 4431--4468.

\bibitem{BWZ} \textsc{C. B\"ohm, M. Wang, W. Ziller}, {\it A variational approach for compact homogeneous Einstein manifolds}, Geom. funct. anal. {\bf 14} (2004), 681--733.

\bibitem{Go} \textsc{I. Goldbring}, {\it Hilbert's fifth problem for local groups}, Ann. of Math. \textbf{172} (2010), 1269--1314.

\bibitem{GZ} \textsc{K. Grove, W. Ziller}, {\it Cohomogeneity one manifolds with positive Ricci curvature}, Invent. math. \textbf{149} (2002), 619 -- 646.

\bibitem{Heb} \textsc{J. Heber}, {\it Noncompact homogeneous Einstein spaces}, Invent. Math. {\bf 133} (1998), 279--352.

\bibitem{Hel} \textsc{S. Helgason}, Differential Geometry, Lie Groups and Symmetric Spaces, {\it Academic Press}, {\it New York}, 1978.

\bibitem{Lau1} \textsc{J. Lauret}, {\it Convergence of homogeneous manifold}, J. Lond. Math. Soc. {\bf 86} (2012), 701--727.

\bibitem{Lau2} \textsc{J. Lauret}, {\it Geometric flows and their solitons on homogeneous spaces}, Rend. Semin. Mat. Univ. Politec. Torino \textbf{74} (2016), 55--93.

\bibitem{Mos} \textsc{G.D. Mostow}, {\it The extensibility of local Lie groups of transformations and groups on surfaces}, Ann. of Math. \textbf{52}, (1950), 606--636. 

\bibitem{OV} \textsc{A. N. Onishchik, E. B. Vinberg}, Lie groups and Algebraic groups, {\it Springer-Verlag}, {\it Heidelberg}, 1990.

\bibitem{Ped} \textsc{F. Pediconi}, {\it A local version of the Myers-Steenrod Theorem}, \texttt{arxiv:1906.02988}.

\bibitem{PaSa} \textsc{J.-S.Park, Y. Sakane}, {\it Invariant Einstein metrics on certain homogeneous spaces}, Tokyo J. Math. {\bf 20}, (1997), 51--61.

\bibitem{PoSp} \textsc{F. Podest{\aac}, A. Spiro}, Introduzione ai Gruppi di Trasformazioni, {\it Volume of the Preprint Series of the Mathematics Department V. Volterra of the University of Ancona}, {\it Ancona}, 1996.

\bibitem{Sp} \textsc{A. Spiro}, {\it A remark on locally homogeneous Riemannian spaces}, Results Math. \textbf{24} (1993), 318--325.

\bibitem{WZ} \textsc{M. Wang, W. Ziller}, {\it Existence and non-existence of homogeneous Einstein metrics}, Invent. math. {\bf 84} (1986), 177--194. 

\end{thebibliography}
\end{document}